\documentclass[11pt]{amsart}
\usepackage{graphicx} % Required for inserting images
\usepackage{amssymb}
\usepackage{amsmath}
\usepackage{stmaryrd}
\usepackage{blindtext}
\usepackage{amsthm}
\usepackage[hidelinks]{hyperref}
\usepackage[english]{babel}
\usepackage{enumitem}
\usepackage[all]{xy}
\usepackage{geometry}
\DeclareMathOperator{\C}{\mathbb C}

\DeclareMathOperator{\Z}{\mathbb Z}
\DeclareMathOperator{\N}{\mathbb N}

\DeclareMathOperator{\g}{\mathfrak g}
\DeclareMathOperator{\h}{\mathfrak h}

\DeclareMathOperator{\ssl}{\mathfrak{sl}}

\DeclareMathOperator{\ad}{\text{ad}}

\DeclareMathOperator{\Id}{\text{Id}}

\DeclareMathOperator{\wt}{\text{wt}}
\DeclareMathOperator{\Aut}{\text{Aut}}

\newcommand{\co}{\left[}
\newcommand{\cf}{\right]}
\newcommand{\po}{\left(}
\newcommand{\pf}{\right)}

\newtheorem{theoreme}{Theorem}[section]

\newtheorem{lemme}[theoreme]{Lemma}
\newtheorem{proposition}[theoreme]{Proposition}
\newtheorem{definition}[theoreme]{Definition}

\newtheorem{remark}[]{Remark}

\title[Coproduct of Drinfeld-Cartan series in type A]{Coproduct of modified Drinfeld-Cartan series for Yangians and quantum affine algebras in type A}
\author{Jérôme Milot}
\date{}
\address{Laboratoire Paul Painlevé, Université de Lille, France}
\email{jerome.milot@univ-lille.fr}

\begin{document}

\begin{abstract}
    We give explicit formulas for the coproducts of modified Drinfeld-Cartan generating series for the Yangian in type $A$ and for the quantum affine algebras in the particular type $A_2$. As an auxiliary result of the latter, we give an explicit presentation of positive prefundamental representations of the quantum affine algebra in the particular type $A_2$. 
\end{abstract}

\maketitle

\tableofcontents{}

\section{Introduction}
Let $\g$ be a complex finite-dimensional simple Lie algebra. For simplicity, we assume $\g$ to be simply laced. Let $I$ the set of its Dynkin nodes, $C = (c_{ij})$ its Cartan matrix and $q$ a non zero complex number which is not a root of unity. Let $Y(\g)$ be the Yangian and $U_q(\hat\g)$ the quantum affine algebra constructed by Drinfeld \cite{Drinfeld85} as Hopf algebra deformations of the current algebra $\g[z]$ and the loop algebra $\g[z,z^{-1}]$ respectively. 

In 1988, Drinfeld gave a new realization of these quantum groups, which is more adapted to study the representation theory \cite{NewDrinfeld}. In this new presentation, the generators of the quantum affine algebra $U_q(\hat \g)$ are encoded in the generating series $x^\pm_i(z), \phi^\pm_i(z)$ for $i \in I$. As an example of commutating relations of these generating series, we have 
\[  \phi_i^-(z) x_j^\pm(w) = \frac{q^{\pm c_{ij}}w-z}{w - q^{\pm c_{ij}}z}\phi^-(z)x_j^\pm(w) \quad \text{ for } i,j \in I. \]
If this new presentation can be equipped with a Hopf algebra structure with the isomorphism between these realizations, the corresponding coproduct is too complicated. A better understanding of these coproducts would lead to explicit ways of studying the representation theory of these quantum groups, notably R-matrices. In this paper, we establish explicit formulas for coproducts for families of Drinfeld new generators in type A which have simpler commutation relations with the other generators.

These generating series arise in the study of the representation theory of quantum affine algebras and their connection with quantum integrable systems. In this context, Baxter introduced commuting operators, called transfer matrices, whose common eigenvalues encode important information about the model. The study of eigenvalues led Frenkel-Reshetikhin \cite{Frenkikhin} to introduce the notion of $q$-character, a quantum analog of classical characters. They conjectured that these eigenvalues can be obtained by computing the q-character of an associated representation and performing a suitable substitution. This conjecture was proved by Frenkel-Hernandez \cite{Frenkandez} using a particular family of representations of the Borel subalgebra, called prefundamental representations. These were introduced by Hernandez-Jimbo \cite{HernandezJimbo} and constructed as limits of Kirillov-Reshetikhin modules. It turns out that the transfer matrices associated with these prefundamental modules play a key role in the substitution procedure described above.

In the same article, Frenkel-Hernandez defined a new family of generating series of the Drinfeld-Cartan subalgebra as the limit of transfer matrices associated with a prefundamental representation. They proved that these series, called \textbf{T-series}, are polynomial on any finite-dimensional representation of $U_q(\hat\g)$. At the level of the algebra, these modified Drinfeld-Cartan series behave particularly well with the other generators. Indeed, for $i,j \in I$:
\[ T_i(z) x_j^-(w)T_i(z)^{-1} = (1 - \delta_{i,j}wz)x_j^-(w)\]
which is a simpler commutation relation than the one for $\phi_i^\pm(z)$. This is one of our motivations to work with these series instead of the standard Drinfeld-Cartan generators. In 2024, Zhang \cite{Zhang} proved that the coproduct T-serias have a compact formula
 \[ \Delta \po T_i(z) \pf = \po 1 \otimes T_i(z) \pf \times  \Theta_i(z) \times \po T_i(z) \otimes 1 \pf \] where each weight component of $\Theta_i(z)$ is polynomial in $z$. 

    \medskip
 
 One of the main results of this article is an explicit formula of the $\Theta_i(z)$ in type $A_2$ (theorem \ref{coproductTseries}):
 \begin{theoreme}
 The Theta series for the quantum affine algebra $U_q(\hat\ssl_3)$ are given by the following two formulas 
      \begin{gather*}
        \Theta_1(z) =\exp_q\po (q-q^{-1})x_{1,0}^-\otimes x_{1,-1}^+z \pf \exp_q\po (q-q^{-1}) \co x_{1,0}^-,x_{2,0}^- \cf_q \otimes \co x_{2,0}^+,x_{1,{-1}}^+ \cf_{q^{-1}} z \pf; \\
          \Theta_2(z) =  \exp_q\po (q-q^{-1}) x_{2,0}^-\otimes x_{2,-1}^+ z \pf  \exp_q\po (q-q^{-1}) \co x_{2,0}^-,x_{1,0}^- \cf_q \otimes \co x_{1,0}^+,x_{2,-1}^+  \cf_{q^{-1}} z \pf
     \end{gather*}
     where $\exp_q$ is a q-exponential and $\co x,y \cf_q = xy - qyx$ is a q-commutator.
 \end{theoreme}

 To prove this result, we take advantage of the universal R-matrix and apply it to a prefundamental module over the Borel subalgebra to determine $\Theta_i(z)$. As an auxiliary result, we compute explicitly a basis for a positive prefundamental module of the Borel subalgebra. These modules have already been studied in type $A_2$ by Bazhanov-Hibberd-Khoroshkin \cite{BHK}, and in type $A_n$ for any $n \in \N$ Jang-Kwon-Park \cite{JKP}. Our contribution gives the explicit action of all root vectors on this module.

\medskip

In 2024, Zhang defined an analog of T-series for Yangians, which also behave well with other generators. Note that, at the level of representations, an analog of the T-series has been studied in 2020 by Gautam-Wendlandt \cite{gautamwendlandt} and Hernandez-Zhang \cite{HernandeZhang} independently. These modified Drinfeld-Cartan series can also be factorized in terms of Theta series in essentially the same way, with the same property of local polynomiality. Our second main result is an explicit formula for the $\Theta_i(z)$ in type $A_n$ for any $n\in \N$ (theorem \ref{coproduitSserie}):

\begin{theoreme}
    For $n \in \N$, the Theta series for the Yangian $Y(\ssl_{n+1})$ are given by the following $n$ formulas, for $i \in \{1,\dots,n\}$:
    \[ \Theta_i(z) = \exp \po \sum\limits_{1 \leqslant j \leqslant i < k \leqslant n+1} E_{kj}\otimes E_{jk} \pf \]
    where $E_{jk}$ are elementary matrices in $\ssl_{n+1} \subset Y(\ssl_{n+1})$.
\end{theoreme}

Note that the result is better than expected: the weight components of the Theta series do not depend on $z$.

If the result is somewhat similar to the quantum affine algebra case, the arguments involved are different. For the Yangian, the $\Theta_i(z)$ can be interpreted as a homomorphism between some modules over shifted Yangians. The commutation rules between this morphism and the action of some generators give a system of equations with a unique solution that we solve to determine $\Theta_i(z)$.

\medskip 

Our result, in addition to providing coproduct formulas for Drinfeld-Cartan generators, could be used in future works to explicitly compute new R-matrices as in \cite[sections 6 and 10]{Zhang} or to compute coproduct formulas for other modified Drinfeld-Cartan series, notably the truncation series of Gerasimo-Kharchev-Lebedev-Oblezin \cite{GKLO}; the case $\ssl_2$ was studied previously in \cite{Milot}.

\medskip

This paper is organized as follows. In section 1, we recall definitions and properties about Yangians, S-series and Theta series. In section 2, we establish the coproduct formulas for the S-series. A part of the proof, which uses some definitions such as shifted Yangians that will not be recalled in section 1 for the sake of simplicity, is given in the appendix. In section 3, we recall definitions and properties about quantum affine algebras, Borel subalgebra and T-series. In section 4, we introduce the notion of root vectors and the universal R-matrix. We compute the root vectors in the case $\g = \ssl_3$, following the construction of Damiani \cite{Damiani}. In section 5, we provide an explicit basis for the first positive prefundamental representation of $U_q(\mathfrak b)$ for $\g = \ssl_3$, using the basis of subalgebras of $U_q(\mathfrak b)$ formed of root vectors determined in section 4. In section 6, we establish the coproduct formulas for T-series, applying the R-matrix to the prefundamental module. In the appendix, we recall the basic definitions and properties of shifted Yangians and give a proof of the admitted result from section 2.

\subsection*{Acknowledgement}
The author acknowledges the support of the CDP C2EMPI, as well as the French State under the France-2030 programme, the University of Lille, the Initiative of Excellence of the University of Lille, the European Metropolis of Lille for their funding and support of the R-CDP-24-004-C2EMPI project.

The author thanks Huafeng Zhang for fruitful discussions, his insightful comments, and careful reading of the manuscript.

\section*{Notations and conventions}
Note $\N := \Z_{\geq0}, \N^* := \N\setminus \{0\}$. Fix $n \in \N^*$. Let $\g = \ssl_{n+1}$ be the special linear Lie algebra, $\h$ its Cartan-subalgebra and $I= \{1,\dots,n\}$ the set of Dynkin nodes. Let $\Delta := (\alpha_i)_{i\in I}$ be the simple roots and denote by $R$ the set of positive roots. Let $(~,~):\h \times \h^*$ be a non-degenerate symmetric bilinear form. The Cartan matrix $C$ is denoted by 
$c_{i,j} = \frac{2(\alpha_i,\alpha_j)}{(\alpha_i,\alpha_i)}$. In type $A$, this matrix is symmetric. Let the fundamental coweights $(\omega_i^\vee)$ be the basis of $\h$ dual to the basis $(\alpha_i)$, that is, $\langle \alpha_i,\omega_i^\vee \rangle = \delta_{ij}$. Denote $\alpha_{i}^\vee = \sum\limits_{j \in I} c_{ij} \omega_j^\vee$.
\[   P^\vee := \bigoplus_{i\in I} \Z \omega_i^\vee; \quad Q := \bigoplus\limits_{i\in I} \Z\alpha_i; \quad Q_+ := \sum_{i\in I} \N \alpha_i\]
the coweight lattice, the root lattice and the positive root cone. For $i \in I$, define the reflections $(s_i)_{i\in I}$ acting on $Q$ by $s_i(x) := x - \langle x,\alpha_i^\vee \rangle\alpha_i$ for $x \in Q$.

Denote by $\hat\g$ the affine Kac-moody algebra associated with $\g$ and $\hat I = \{0,\dots,n\}$ the set of affine Dynkin nodes. Let $q \in \C^\times$ not be a root of unity. Define for $m \in \Z$ and $k \in \N$:
\[ (m)_q := \frac{q^{2m} - 1}{q^2-1}; \quad (m)_q! := \prod\limits_{s=1}^m [s]_q; \quad [m]_q := \frac{q^m-q^{-m}}{q-q^{-1}}; \quad \left[ \begin{array}{c}
     m  \\
     k 
\end{array}\right]_q := \prod\limits_{s=1}^k [m-s+1 ]_q.\]

Define the quantum Cartan matrix $C(q) = \po [c_{i,j}]_q\pf_{i,j \in I}$. Since $q$ is generic, it is invertible. Denote by $\tilde{C}(q) = \po \tilde C_{i,j}(q) \pf_{i,j \in I}$ its inverse.

\section{Yangians, S-series, Theta series}
In this section, we recall two presentations of the Yangian: the J-presentation and the current presentation. The first one will be useful to simplify some of our computations, the second is the one we will work with. We then introduce series, called S-series, with coefficients generating an important commutative subalgebra: the Drinfeld-Cartan subalgebra. An explicit formula for the coproduct of the standard generators of the Drinfeld-Cartan subalgebra remains open. It has been proved in \cite{Zhang} that the S-series have convenient coproduct formulas, using Theta series which are series with coefficients in the Drinfeld-Cartan subalgebra that can be interpreted as associators for some modules of the Yangian. 

\subsection{Yangians}
In this subsection, we recall the J-presentation, current presentation of Yangian and the isomorphism between these two presentations.

\begin{definition}\cite{Drinfeld85}
    Let $(\cdot,\cdot)$ be an invariant bilinear form on $\g$ and $\po x_\lambda \pf$ be an orthonormal basis of $\g$ with respect to $(\cdot,\cdot)$. The \textbf{J-presentation} of the Yangian $Y^J$ is the algebra defined by generators $x,J(x)$ for $x \in \g$ that satisfy the following relations for $x,y,z \in \g$:
    \begin{gather} xy - yx = \co x,y \cf_{x,y\in \g}; \quad \co x, J(y) \cf = J(\co x,y \cf); \\
    J(ax+by) = aJ(x) + bJ(y); \\
    \co J(x), \co J(y ),z\cf \cf - \co x, \co J(y),J(z) \cf \cf = \sum\limits_{\lambda,\mu,\nu} \po \co x, x_{\lambda} \cf , \co \co y, x_{\mu} \cf, \co z, x_\nu\cf\cf \pf \{ x_{\lambda}, x_\mu, x_\nu\}
    \end{gather}
    where $\{x_1,x_2,x_3\} = \frac{1}{24}\sum\limits_{\sigma\in\mathfrak S_3} x_{\sigma(1)}x_{\sigma(2)}x_{\sigma(3)}$.
\end{definition}

\begin{remark}
    The relation $(3)$ is not exactly the same in the particular case $A_1$ and should be replaced by:
    \[ \co \co J(x_1),J(x_2) \cf, \co x_3,J(x_4) \cf \cf + \co \co J(x_3),J(x_4) \cf, \co x_1,J(x_2)  \cf \cf = \sum\limits_{\lambda,\mu,\nu} \alpha_{\lambda,\mu,\nu} \{ x_{\lambda},x_{\mu},J(x_{\nu})\} \]
    where $\alpha_{\lambda,\mu,\nu} = \po \co x_1,x_\lambda \cf, \co \co x_2, x_\mu \cf, \co \co x_3,x_4 \cf,x_\mu \cf  \cf \pf + \po \co x_4,x_\lambda \cf, \co \co x_3, x_\mu \cf, \co \co x_1,x_2 \cf,x_\mu \cf  \cf \pf$.
\end{remark}

The J-presentation was introduced in 1985 by Drinfeld. It has a simple Hopf algebra structure whose coproduct is given by, for $x \in \g$:
\[
     \Delta(x) = x \otimes 1 + 1 \otimes x; \quad \Delta \po J(x) \pf = J(x) \otimes 1  + 1 \otimes J(x) + \frac{1}{2}\co x\otimes 1, \Omega \cf
\]
where $\Omega$ is the quadratic Casimir element in the adjoint representation of $\g$.

Now, let us define the current presentation, which is the one we will mainly work with. It was introduced by Drinfeld \cite{NewDrinfeld} in 1988. It is often more convenient to work algebraically with this presentation rather than with the J-presentation.

\begin{definition}\cite{NewDrinfeld}
The \textbf{current presentation of the Yangian} $Y$ associated with $\g$ is the associative algebra defined by the generators $x^-_{i,m},x^+_{i,m},\xi_{i,p}$ for all $i \in I$, $m, p \in \N$ satisfying the following relations:
    \begin{gather*}
        \co \xi_{i,p},\xi_{j,m}\cf = 0, \quad \co x^+_{i,p},x^-_{j,m} \cf = \delta_{i,j}\xi_{i,p+m}, \\ 
        \co \xi_{i,m+1},x^\pm_{j,p} \cf - \co \xi_{i,m},x^\pm_{j,p+1} \cf = \pm \frac{1}{2}c_{ij} \po \xi_{i,m} x^\pm_{j,p} + x^\pm_{j,p}\xi_{i,m} \pf,  \\
        \co x^\pm_{i,m+1},x^\pm_{j,p} \cf - \co x^\pm_{i,m},x^\pm_{j,p+1} \cf = \pm\frac{1}{2}c_{ij} \po x^\pm_{i,m} x^\pm_{j,p} + x^\pm_{j,p} x^\pm_{i,m} \pf, \\
        \ad^{1-c_{i,j}}_{x_{i,0}^\pm}(x_{j,0}^\pm) = 0 \quad \quad (i\neq j) 
    \end{gather*}
    where $\ad_x(y) = xy - yx$.

    The Yangian admits a $Q$-grading, called \textbf{weight grading} defined on the generators by $\wt(x_{i,m}^\pm) = \pm\alpha_i;$ and $\wt(\xi_{i,m}) = 0$. For $\beta \in Q$, denote by $Y_\beta$ the subspace of elements in $Y$ of weight $\beta$.

    Define the generating series for $i \in I$:
\[ \xi_i(z) := 1 + \sum\limits_{m \geq 0} \xi_{i,m} z^{-m-1}; \quad x^\pm_i(z) = \sum\limits_{m \geq 0}x_{i,m}^\pm z^{-m-1} .\]
These are power series in $z^{-1}$ of leading terms $1$ and $x_{i,0}^\pm z^{-1}$, respectively. The $\xi_{i,p}$ for $i \in I, p\in \N$ generate a commutative subalgebra of $Y(\g)$, denoted by $Y^0$, called the \textbf{Drinfeld-Cartan subalgebra}.
\end{definition}

For $\alpha \in R$, let $x_\alpha^\pm \in \g$ of weight $\pm \alpha$ be such that $(x_\alpha^+,x_\alpha^-) = 1$ and $x^\pm_{\alpha_i}=x_{i,0}^\pm$. Let us give the isomorphism between these two presentations.

\begin{theoreme}\cite[Theorem 1]{NewDrinfeld}\cite[Theorem 2.6]{GRW}\label{isopresentations}
The algebras $Y^J$ and $Y$ are isomorphic and an isomorphism $\varphi : Y^J \to Y$ between the two presentations is provided by:
\begin{gather*} \varphi(h_i) = \xi_{i,0}; \quad \varphi\po J(h_i)\pf = \xi_{i,1} + \frac{1}{4} \sum\limits_{\alpha \in R} (\alpha,\alpha_i) \{x_\alpha^+,x_\alpha^-\} - \frac{1}{2} \xi_{i,0}^2; \\
\varphi(x_i^\pm) = x_{i,0}^\pm; \quad \varphi(J(x_{i}^\pm)) =  x_{i,1}^\pm \pm \frac{1}{4} \sum\limits_{\alpha \in R}\left\{ \co x_i^\pm,x_\alpha^\pm \cf,x_\alpha^\mp \right\} -\frac{1}{4} \{x_{i,0}^\pm,\xi_{i,0}\}.
\end{gather*}
In particular, for $i \in I$, we have:
$$\varphi(J(x_{i,0}^+)) = x_{i,1}^+ + \frac{1}{4} \po \sum\limits_{k = i+1}^n \po E_{i,k}E_{k,i+1} + E_{k,i+1}E_{i,k} \pf - \sum\limits_{j=1}^{i} \po  E_{j,i+1} E_{i,j} + E_{i,j}E_{j,i+1}\pf \pf$$
where $E_{j,k}$ are the elementary matrices in $\ssl_{n+1} \subset Y(\ssl_{n+1})$. 
\end{theoreme}

The current presentation of the Yangian can be equipped with a Hopf algebra structure by defining $\varphi$ as a Hopf algebra morphism. However, a computable formula for the coproduct of the Drinfeld-Cartan generators remains an open question. For example, for $\ssl_2$ \cite[Definition 2.24]{Molev}:
\[ \Delta(\xi(z)) = \left( 1 \otimes \xi(z) \right) \times  \left( \sum\limits_{k=0}^\infty(-1)^k(k+1)(x^-(z+1))^k \otimes (x^+(z+1))^k \right) \times  \left( \xi(z) \otimes 1 \right).\]

\subsection{S-series and Theta series}
The coproduct of standard Drinfeld-Cartan generators is quite complicated to manipulate algebraically. In 2024, Zhang introduced a new family of Drinfeld-Cartan generators, the S-series \cite{Zhang}, to construct R-matrices for 1-dimensional representations of Yangians. These series are defined by solving a difference equation in the Drinfeld-Cartan subalgebra.

In his article, Zhang proved that the coproducts of the S-series have a more compact formula in terms of Theta series. These series can be interpreted as associators for shifted Yangians, a family of associative algebras introduced by Brundan-Kleshchev \cite{BrundanKleshchev} in type A and \cite{GKLOtermes, BFN} in full generality that generalizes standard Yangians. One of the main results of this article is an explicit formula for the Theta series and the coproduct of S-series in type A.

First, to introduce S-series, we need the series of Gerasimov–Kharchev–Lebedev–Oblezin \cite[lemma 2.1]{GKLO}, called GLKO-series. These are Laurent series $A_i(z) \in 1 + z^{-1}Y^0[[z^{-1}]]$ for $i \in I$ uniquely determined by the following difference equations for $i \in I$:
    \[ \xi_i(z) = \frac{A_{i-1}(z-\frac{1}{2})A_{i+1}(z-\frac{1}{2})}{A_i(z)A_i(z-1)} \in 1 + z^{-1}Y^0[[z^{-1}]]\]
    with $A_{-1}(z) = A_{n+1}(z) = 1$. Introduce the GKLO-elements $$A_i(z) = \exp\po \sum\limits_{m \geq 0} a_{i,m}z^{-m-1} \pf.$$
\begin{remark}
    The GKLO-series have their own interest: some quotients of Yangians, called truncated Yangians, arose from the GKLO series. An explicit formula for their coproduct is expected to be relatively simple (for example, it is in type $A_1$ \cite{Milot}). To compute it, one can compute the formulas for the GKLO-series from the coproducts of the S-series.
\end{remark}

Now, let us recall the S-series from \cite[proposition 4.1]{Zhang}. For $i \in I$, there is a unique power series $S_i(z)$ in $1 + z^{-1}Y^0[[z^{-1}]]$ satisfying the difference equation 
    \[ S_i(z+1) = S_i(z)\times A_i(z) \times \exp \po a_{i,0} \sum\limits_{k>0}\frac{(-z)^{-k}}{k} \pf.\]
Note that the S-series behave particularly well with Drinfeld generators \cite[proposition 4.1]{Zhang}. For $i \neq j \in I$, the series $S_i(z)$ commutes with $x_{j,m}^\pm$ for $m \in \N$, and for $j = i$ we have:
    \[
        S_i(z)x_{i,m}^-S_i(z)^{-1} = x_{i,m}^- - x_{i,m+1}^-z^{-1}; \quad
        S_i(z)^{-1}x_{i,m}^+S_i(z) = x_{i,m}^+ - x_{i,m+1}^+z^{-1}.
    \]
    
\begin{remark}
    In \cite[Theorem 4.6]{Zhang}, a T-operator is defined as the composition of the S-series with a (quite simple) linear automorphism. At the level of representations, this T-operator is thought to be similar to the T-series introduced in \cite{Frenkandez} in the case of quantum affine algebras. It can be used to construct new R-matrices \cite[Section 6]{Zhang}. Yet, in this article, we are mainly interested in the algebra level, so we will not say more about these operators.
\end{remark}

Now, let us discuss the coproduct of the S-series. In \cite{Zhang}, it was proved that the coproduct has a compact form as follows. For $i \in I$:
\[ \Delta(S_i(z)) = \po 1 \otimes S_i(z) \pf \times \Omega_i(z) \times \po S_i(z) \otimes 1 \pf \]
    with $\Omega_i(z) \in 1\otimes 1 + z^{-1} \po Y\otimes Y\pf[[z^{-1}]] $. 
    
    For $\beta \in Q_+$, define $\Omega_{i,\beta}(z) \in \po Y_{-\beta}\otimes Y_\beta \pf[[z^{-1}]]$ as the $(-\beta,\beta)$-component of $\Omega_i(z)$. Then $\Omega_i(z) = \sum\limits_{\beta \in Q_+}\Omega_{i,\beta}(z)$ and, according to lemma 5.1 \cite{Zhang}, $\Omega_{i,0} = 1\otimes 1$.

Now, let us introduce the \textbf{Theta series} \cite[Definition 5.2]{Zhang}:
\[ \forall \beta \in Q_+: \quad \Theta_{i,\beta}(z) = \Omega_{i,\beta}(z) (-z)^{\beta_i := \langle \omega_i^\vee , \beta \rangle}. \] Note that $\Theta_{i,0}(z) = 1\otimes 1$. Let us then consider the following series:
\[ \Theta_i(z) = \sum\limits_{\beta \in Q_+} \Theta_{i,\beta}(z) \in \prod\limits_{\beta\in Q_+} \po Y_{-\beta}\otimes Y_\beta \pf((z^{-1})).\]

The interest of $\Theta_i(z)$ is that it can be interpreted as associators for modules over shifted Yangians (see \cite[section 5.2]{Zhang}). From this property, one obtains the \textbf{intertwining relations}: a system of equations that admits a unique solution, which is the Theta series. These relations, which can be quite complicated to write, have an important corollary: the weight components $\Theta_{i,\beta}(z)$ are polynomial in $z$. However, in type A, these relations are computable.

\section{Formula for the coproduct of S-series in type A}
We establish the main result of this paper for the Yangian: an explicit formula for the coproduct of the S-series. We use the fact that the Theta series are module homomorphisms, so that they must commute with the action of the generators. It turns out that this commutation property in type A gives a computable system of $n$ equations that determines $\Theta_i(z)$. Solving this system gives a simple expression of the Theta series, depending on elementary matrices, from which we deduce the formula for the coproduct of the S-series.

The first main result of this paper is the formula for the Theta series. 
\begin{theoreme}\label{coproduitSserie}
    For $n \in \N$, the Theta series for the Yangian $Y(\ssl_{n+1})$ are given by the following $n$ formulas, for $1 \leqslant i \leqslant n$:
    \[ \Theta_i(z) = \exp\left( \sum\limits_{1\leqslant j \leqslant i < k \leqslant n+1} E_{kj}\otimes E_{jk} \right) \]
    from which one can deduce the formulas for the coproduct of the S-series:
    \[ \Delta(S_i(z)) =  (1 \otimes S_i(z) )\exp \po- \sum\limits_{1 \leqslant j \leqslant i < k \leqslant n+1} E_{kj}\otimes E_{jk}  z^{-1}\pf (S_i(z) \otimes 1)\]
    where $E_{j,k}$ are the elementary matrices of $\ssl_{n+1} \subset Y(\ssl_{n+1})$.
\end{theoreme}

\begin{remark}
    The expression of $\Theta_i(z)$ was known to be locally polynomial in $z$, but the result is even more convenient since it does not depend on $z$. Moreover, note that the terms in the sum commute.
\end{remark}

We need some preliminary results to establish this result. Let us begin with the system of equations that the Theta series $\Theta_i(z)$ have to satisfy.

\begin{proposition}\label{intertwiningyangian}
    For $i \in I$, the Theta series $\Theta_i(z)$ is the unique series in $\prod\limits_{\beta\in Q_+} \po Y_{-\beta}\otimes Y_\beta \pf((z^{-1}))$ with the $(0,0)$-weight component equal to $1 \otimes 1$ that satisfies the following $n$ equations for $j \in I$:
    \begin{gather*}
        \text{For } j \neq i: \quad \co x_{j,0}^+ \otimes 1 + 1 \otimes x_{j,0}^+, \Theta_i(z) \cf = 0; \\
        \co x_{i,0}^+ \otimes 1 + 1\otimes x_{i,1}^+ -  z (1 \otimes x_{i,0} ^+) , \Theta_i(z)  \cf =  \Theta_i(z) \po   \sum\limits_{j= 1}^{i} E_{i,j} \otimes E_{j,i+1}  - \sum\limits_{k=i+1}^{n+1}E_{k,i+1}\otimes E_{i,k} \pf.
    \end{gather*}
\end{proposition}

\begin{remark}
    This proposition essentially follows from the proof of \cite[Theorem 5.7]{Zhang}. For completeness, we will provide more details on the intertwining property and the uniqueness in the appendix.
\end{remark}

Let us solve these equations in order to prove the theorem. We need the following results. First, we introduce a completion of $Y \otimes Y$:
\[ Y \otimes_z Y := \sum\limits_{\alpha,\beta \in Q} \po \prod\limits_{\gamma \in Q_+} \po Y_{(\alpha-\gamma)} \otimes Y_{(\beta+\gamma)} \pf((z^{-1})) \pf\]
which is an algebra. Note that $\Theta_i(z) \in Y\otimes_z Y$.

\begin{lemme}\label{commutators}
    Fix $i\in I$. Let $1 \leqslant j \leqslant i < k \leqslant n+1$. Denote $y = \sum\limits_{1 \leqslant j \leqslant i < k \leqslant n+1} E_{k,j}\otimes E_{j,k} \in Y\otimes Y$. In $Y \otimes_z Y$, we have the following relations:
    \begin{align*} 
    \co 1\otimes x_{i,0}^+, \exp(y) \cf &= 0; \\
    \co x_{i,0}^+ \otimes 1, \exp(y) \cf &= \exp(y) \po \sum\limits_{j= 1}^i E_{i,j} \otimes E_{j,i+1} - \sum\limits_{k=i+1}^{n+1}E_{k,i+1}\otimes E_{i,k} - \sum\limits_{\substack{1 \leqslant j \leqslant i \\ i < k \leqslant n+1}} E_{k,j} \otimes E_{j,i+1}E_{i,k} \pf \\
    \co x_{i,1}^+, E_{j,k} \cf &= E_{i,k}E_{j,i+1}; \\
    \co 1 \otimes x_{i,1}^+, \exp(y)  \cf &= \exp(y)\po\sum\limits_{1 \leqslant j\leqslant i < k \leqslant n+1} E_{k,j} \otimes E_{j,i+1}E_{i,k} \pf.
    \end{align*}
\end{lemme}

\begin{proof}    
    First, note that $\exp(y) \in \prod\limits_{\gamma \in Q_+}Y_{-\gamma} \otimes Y_{\gamma} \subset Y\otimes_z Y$.
    
    For the first relation, note that $\co 1 \otimes x_{i,0}^+,y \cf = 0$. Hence, the formula follows from the derivation property of the Lie bracket.

    Denote $x = x_{i,0}^+ \otimes 1$. For the second relation, we have:
    \[ \exp(-y)x\exp(y) = \exp(\ad_{-y})(x) = \sum\limits_{k=0}^\infty \frac{1}{k!} \ad_{-y}^k(x)\]
    where $\ad_{-y}(x) = \co x,y \cf$. Now, let us compute the first terms of this sum:
    \begin{gather*}
         \co x,y \cf = \sum\limits_{j= 1}^i E_{i,j}\otimes E_{j,i+1} - \sum\limits_{k = i+1}^{n+1}E_{k,i+1}\otimes E_{i,k}; \\
        \co \co x,y \cf,y \cf = - 2\sum\limits_{1 \leqslant j \leqslant i < k \leqslant n+1} E_{k,j} \otimes E_{j,i+1}E_{i,k}; \\
        \co \co \co x,y \cf,y \cf,y \cf = 0.
    \end{gather*}
    Then, for $k > 2$, we have $\ad_{-y}^k(x) = 0$. One can deduce
    \begin{align*} \co x, \exp(y) \cf &= \exp(y)\po \exp(-y)x\exp(y) - x \pf = \exp(y) \po \ad^0_{-y}(x) + \ad_{-y}(x) + \frac{1}{2}\ad_{-y}^2(x) - x \pf \\&= \exp(y) \po  \co x,y \cf + \frac{1}{2}\co \co x,y \cf, y\cf\pf 
    \end{align*}
    which is exactly the second formula.

    For the third, we use the defining relations of the J-presentation and deduce the relations for the current presentation using the isomorphism $\varphi$ introduced in theorem \ref{isopresentations}. Indeed, on the one hand, for $1 \leq j \leq i$ and $i < k \leq n+1$:
    \[ \co J(x_{i,0}^+),E_{j,k} \cf = J\po \co x_{i,0}^+,E_{j,k} \cf \pf = 0.\]
    On the other hand, by theorem \ref{isopresentations} we have:
    \[ \varphi(J(x_{i,0}^+)) = x_{i,1}^+ + \frac{1}{4} \po \sum\limits_{k = i+1}^{n+1} \po E_{i,k}E_{k,i+1} + E_{k,i+1}E_{i,k} \pf - \sum\limits_{j=1}^{i} \po  E_{j,i+1} E_{i,j} + E_{i,j}E_{j,i+1}\pf \pf. \]
    Then, for $1 \leqslant j \leqslant i < k \leqslant n +1$:
    \[ 0 = \co \varphi(J(x_{i,0}^+)),E_{j,k} \cf = \co x_{i,1}^+,E_{j,k}  \cf - E_{i,k}E_{j,i+1}. \]

    Finally, for the fourth relation, note that
    \begin{gather*} \co 1 \otimes x_{i,1}^+, y \cf = \sum\limits_{1 \leq j \leq i < k\leq n+1} E_{k,j} \otimes {E_{i,k}E_{j,i+1}}; \\  
    \co \co 1\otimes x_{i,1}^+,y \cf, y   \cf = 0.
    \end{gather*}
    Then, using again the formula used for the second relation, we have:
    \[ \co 1 \otimes x_{i,1}^+, \exp(y) \cf = \exp(y) \co 1\otimes x_{i,1}^+,y \cf\]
    which is the fourth formula.
\end{proof}

We can now prove the theorem.

\begin{proof}[Proof of the theorem \ref{coproduitSserie}]
    Still denote $y = \sum\limits_{1 \leq j \leq i < k \leq n+1} E_{k,j} \otimes E_{j,k}$. Let us check that $\Theta_i(z) = \exp(y)$ satisfies the $n$ relations of proposition \ref{intertwiningyangian}. The last relation is a direct consequence of the lemma \ref{commutators}. Let us check the relations for $j \neq i$.

    Let $j < i$. Then, note that
    \[ \co x_{j,0}^+ \otimes 1,y  \cf =  - \sum\limits_{k = i+1}^{n+1} E_{k,j+1} \otimes E_{j,k}; \quad \co 1 \otimes x_{j,0}^+, y \cf = \sum\limits_{k= i+1}^{n+1} E_{k,j+1}\otimes E_{j,k} \]
    so $\co x_{j,0}^+ \otimes 1 + 1 \otimes x_{j,0}^+, y \cf = 0$ and by the derivation property of the Lie bracket, we have $$ \co x_{j,0}^+ \otimes 1 + 1 \otimes x_{j,0}^+, \Theta_i(z) \cf = 0.$$

    Similarly, for $i < k$:
    \[ \co x_{k,0}^+ \otimes 1, y \cf = \sum\limits_{j=1}^i E_{k,j}\otimes E_{j,k+1}; \quad \co 1 \otimes x_{k,0}^+, y \cf = - \sum\limits_{j= 1}^i E_{k,j}\otimes E_{j,k+1}\]
    hence, as above, $\co x_{k,0}^+ \otimes 1  + 1 \otimes x_{k,0}^+, \Theta_i(z) \cf = 0$. This complete the proof.
\end{proof}

\section{Quantum affine algebras, T-series and Theta series}
We now aim to establish analog formulas for the quantum affine algebra $U_q(\hat \ssl_3)$. The result is similar, but requires different kinds of arguments. The T-series, introduced in \cite{Frenkandez}, have the same kind of factorization formula in terms of Theta series. Here, we will take advantage of the universal R-matrix and positive positive prefundamental modules of the Borel subalgebra. The computations are more ad hoc, which is reflected in the restriction of the result to the particular type $A_2$. However, we expect the formula to be generalizable to type $A_n$ and as in the Yangian situation.

\subsection{Quantum affine algebras}
As for Yangians, we recall two presentations, the Drinfeld-Jimbo and the Drinfeld new realization, and the correspondence between them.

\begin{definition}\cite{Drinfeld85}
    The \textbf{quantum affine algebra} $U_q(\hat\g)$ associated with $\g$ is the associative algebra defined by the generators $E_i,F_i,K_i,K^{-1}_i$ for $i \in \hat{I}$ satisfying the following relations for $i,j \in \hat{I}$:
    \begin{gather*}
        K_iK_i^{-1} = K_i^{-1}K_i = 1; \quad  K_iK_j = K_jK_i \quad K_0K_1\dots K_n = 1  \\
        K_iE_j = q^{c_{ij}}E_jK_i; \quad K_iF_j = q^{-c_{ij}}F_jK_i; \quad \co E_i,F_j \cf = \delta_{ij}\frac{K_i - K_i^{-1}}{q-q^{-1}}; \\
        \sum\limits_{s=0}^{-c_{ij}} \left[ \begin{array}{c} 1 - c_{ij}  \\ s \end{array} \right]_q E_i^{1-c_{ij}-s}E_j E_i^s = \sum\limits_{s=0}^{-c_{ij}} \left[ \begin{array}{c} 1 - c_{ij}  \\ s \end{array} \right]_q F_i^{1-c_{ij}-s}F_j F_i^s = 0 \text{ for } i\neq j
    \end{gather*}
    where the last relations are called the quantum Serre's relations.
\end{definition}

The quantum affine algebra has a Hopf algebra structure, whose coproduct is given by:
\begin{gather*}
    \Delta(E_i) = E_i\otimes 1 + K_i \otimes E_i; \quad \Delta(F_i) = 1 \otimes F_i + F_i \otimes K_i^{-1}; \quad \Delta(K_i) = K_i\otimes K_i.
\end{gather*}

We introduce an automorphism $\Phi$ and an anti-automorphism $\Omega$ of $U_q(\hat{\g})$ defined by, for $i \in \hat{I}$:
\begin{gather*}
    \Phi(E_i) = F_i; \quad \Phi(F_i) = E_i; \quad \Phi(K_i) = K_i; \quad \Phi(q) = q^{-1};\\
    \Omega(E_i) = F_i; \quad \Omega(F_i) = E_i; \quad \Omega(K_i) = K_i^{-1}; \quad \Omega(q) = q^{-1}.
\end{gather*}

We will need to study some representations that are not representations of the whole quantum affine algebra but of the Borel subalgebra.

\begin{definition}\cite{NewDrinfeld}
    The \textbf{upper Borel subalgebra}, denoted $U_q(\mathfrak{b})$, is the Hopf subalgebra of $U_q(\hat\g)$ generated by $E_i,K_i$ for $i \in \hat{I}$. The \textbf{lower Borel subalgebra}, denoted $U_q(\mathfrak c)$, is the Hopf subalgebra of $U_q(\hat\g)$ generated by $F_i,K_i$ for $i \in \hat{I}$.
\end{definition}

Yet, it will be more convenient to work with Drinfeld's new realization of the quantum affine algebra.

\begin{definition}\cite{NewDrinfeld}
The Drinfeld new realization of \textbf{quantum affine algebra} $U_q(\hat\g)$ associated with $\g$ is the associative algebra defined by the generators $x_{i,m}^\pm,\phi_{i,m}^\pm$ for $i \in I, m\in \Z$ satisfying the following relations for $m,p \in \Z$, $i,j \in I$ and $\epsilon \in \{ +,- \}$:
\begin{gather*}
    \phi_{i,m}^+ = 0 \text{ for } m  < 0; \quad \phi_{i,m}^- = 0 \text{ for } m > 0; \quad \phi_{i,0}^+\phi_{i,0}^- = 1; \\
    \co \phi_{i,m}^\pm, \phi_{i,p}^{\epsilon} \cf = 0 ; \co x_{i,m}^+,x_{i,p}^- \cf = \delta_{i,j} \frac{\phi_{i,m+p}^+ - \phi_{i,m+p}^-}{q-q^{-1}};  \\
    \phi^\epsilon_{i,m+1} x^\pm_{j,p} - q^{\pm c_{ij}}\phi^\epsilon_{i,m}x^\pm_{j,p+1} = q^{\pm c_{ij}}x_{j,p}^\pm\phi^\epsilon_{i,m+1} - x_{j,p+1}^\pm\phi_{i,m}^\epsilon; \\
    x^\pm_{i,m+1} x^\pm_{j,p} - q^{\pm c_{ij}}x^\pm_{i,m}x^\pm_{j,p+1} = q^{\pm c_{ij}}x_{j,p}^\pm x^\pm_{i,m+1} - x_{j,p+1}^\pm x_{i,m}^\pm; \\
    \sum\limits_{k = 0}^{1-c_{ij}}(-1)^k \left[ \begin{array}{c} 1 - c_{ij}  \\ k \end{array} \right]_q (x_{i,0}^\pm)^{1 - c_{ij} - k} x_{j,0}^\pm (x_{i,0}^\pm)^k = 0 \text{ if } i\neq j. 
\end{gather*}
    We define the generating series $\phi_i^\pm(z) := \sum\limits_{k\in \Z} \phi^\pm_{i,k}z^k$.
\end{definition}

The isomorphism between these two presentations has been given by Beck \cite{Beck}. Let us give formulas for $U_q(\hat\ssl_3)$: 
    \begin{gather*} K_i^{\pm 1} = \phi_{i,0}^\pm; \quad E_i = x_{i,0}^+; \quad F_i = x_{i,0}^- \text{ for } i \in \{1,2\}
    \\
       E_0 = K_0\co x_{1,1}^-,x_{2,0}^- \cf_q = -K_0\co x_{2,1}^-,x_{1,0}^- \cf_q  ; \quad
    F_0 = \Omega(E_0)= \co x_{2,0}^+,x_{1,-1}^+ \cf_{q^{-1}} K_0^{-1}\end{gather*}
    where $[x,y]_q = xy - qyx$.
    We will see more examples in the next section with the root vectors.

\begin{remark}
    The upper Borel subalgebra $U_q(\mathfrak b)$ contains the generators $x_{i,m}^+, x_{i,r}^-, \phi_{i,m}^+$ for $i \in I, m \geq 0, r > 0$. The lower Borel subalgebra $U_q(\mathfrak c)$ contains the generators $x_{i,-m}^-, x_{i,-r}^+, \phi_{i,-m}^-$ for $i \in I, m \geq 0, r >0$.
\end{remark}

Define the \textbf{Drinfeld-Cartan subalgebra} $U_q^0$ as the subalgebra of $U_q(\hat\g)$ generated by $\phi^\pm_{i,m}$ for $i \in I, m\in \mathbb Z$. Define the Drinfeld-Cartan elements $(h_{i,s})_{s \in \mathbb Z \setminus\{0\}, i \in I}$ by:
\[ \phi_i^\pm(z) = \phi_{i,0}^\pm \exp \po \pm(q-q^{-1}) \sum\limits_{\pm s > 0} h_{i,s}z^s \pf.\]
Note that $\phi_{i,\pm 1}^\pm = \pm(q-q^{-1}) \phi_{i,0}^\pm h_{i,\pm 1}$ and for $s \in \N$, $h_{i,\pm s}$ can be expressed in terms of $\phi^\pm_{i,\pm s}, \phi^\pm_{i,0},h_{i,\pm k}$ for $k < s$.\\
The Drinfeld-Cartan elements commute, and one can check:
\[  x^\pm_{j,m+s} = \frac{s}{\co c_{ij}s \cf_q} \co h_{i,s}, x^\pm_{j,m} \cf; \quad .\]

%The coproducts of Drinfeld-Cartan elements have been studied in \cite{Finkeliuk}. For example, for $\ssl_3$:
%\small{\begin{gather*}
 %   \Delta(h_{1,0}) = \square(h_{1,0}); \quad \Delta(h_{1,0}) = \square(h_{1,0}); \\ 
  %  \Delta(h_{1,-1}) = \square(h_{1,-1}) +  (q-q^{-1})\po \po q + q^{-1} \pf x_{1,0}^- \otimes  x_{1,-1}^+ - x_{2,0}^- \otimes  x_{2,-1}^+ - q^2 \co x_{2,0}^-, x_{1,0}^-  \cf_{q^{-3}} \otimes \co x_{2,0}^+,x_{1,-1}^+ \cf_{q^{-1}} \pf;  \\
%    \Delta(h_{2,-1}) = \square(h_{2,-1}) +  (q-q^{-1}) \po \po q + q^{-1} \pf x_{2,0}^- \otimes  x_{2,-1}^+ - x_{1,0}^- \otimes  x_{1,-1}^+ + q^2 \co x_{1,0}^-, x_{2,0}^-  \cf_{q^{-3}} \otimes \co x_{2,0}^+,x_{1,-1}^+ \cf_{q^{-1}} \pf.
%\end{gather*}}
%where $\square(x) = 1 \otimes x + x \otimes 1$.

\normalsize{}

\subsection{T-series and Theta series}
The T-series were first introduced by Frenkel-Hernandez \cite{Frenkandez} in the study of Baxter’s Q-operators for quantum integrable systems. They are defined as the limit of transfer matrices (which are traces of monodromy matrices) of a positive prefundamental module over the upper Borel subalgebra. In \cite{Zhang}, Zhang proved that the coproduct of T-series can be factorized as in the Yangian situation in terms of Theta series. Furthermore, the Theta series can be computed by the universal R-matrix.

Recall $\tilde{C}(q) = \po \tilde C_{ij}(q) \pf_{i,j\in I}$ the inverse of the quantum Cartan matrix and $[s]_q = \frac{q^s-q^{-s}}{q-q^{-1}}$.

\begin{definition}\cite[proposition 5.5]{Frenkandez}
    The T-series $T_i(z)$ are invertible power series in $z$ with coefficients in the Drinfeld-Cartan subalgebra $U_q^0$ defined by:
    \[ T_i(z) := \exp\po \sum\limits_{s > 0} \sum\limits_{j \in I} \frac{\tilde{C}_{ij}(q^s) h_{j,-s}}{\co  s\cf_q} z^s\pf.\]
\end{definition}

The T-series have the same kind of commutation relations as the S-series, given in the introduction. As for the Yangian, we can define the \textbf{Theta series} by \cite[definition 9.2]{Zhang}:
\[ \Delta(T_i(z)) = \po 1 \otimes T_i(z) \pf \times \Theta_i(z) \times \po T_i(z) \otimes 1 \pf.\]
For $\beta \in Q_+$, denote by $\Theta_{i,\beta}(z) \in \po U_q(\mathfrak c)\otimes U_q(\mathfrak c) \pf[[z]]$ the $(-\beta,\beta)$-component of $\Theta_i(z)$. One can check that we have $\Theta_{i,0}(z) = 1 \otimes 1$ and $\Theta_i(z) = \sum\limits_{\beta \in Q_+} \Theta_{i,\beta}(z)$.

Fix $i \in I$. To compute $\Theta_i(z)$, we follow section 9 of \cite{Zhang}. The weight components of $\Theta_i(z)$ can be expressed in terms of monodromy matrices. These are obtained by applying the universal R-matrix on a $U_q(\mathfrak{b})$-module, called prefundamental. In the next section, we explicitly describe the construction of a  basis of subalgebras of $U_q(\mathfrak b)$ formed of root vectors to express the R-matrix, following \cite{Damiani}.

\section{Root vectors and universal R-matrix}
%From now on, consider $\g$ of type $A_2$. 
We recall the construction of the root vectors of the Borel subalgebra $U_q(\mathfrak{b})$ from \cite{Damiani}. This will be used later to study the Theta series. Using the action of the braid group described in \cite{Beck}, we compute the real positive roots and the related root vectors. Then, we express the universal R-matrix in this basis. 

Consider the set of \textbf{real positive roots} $\hat R$ of the affine Lie algebra $\hat \g$. It is the disjoint union of the two sets 
\[ \hat R_+ :=\{ m\delta + \alpha ~|~\alpha\in R;~m\geq 0\}; \quad \hat R_ - :=\{ r\delta - \alpha ~|~\alpha \in R; r >0 \}.\]
In this section, our goal is to construct the root vectors associated with these real positive roots. 

First, let us recall the basic properties of the braid group and its action on the quantum affine algebra. We will use it to construct the real positive roots and the associated root vectors.

Consider $W_0$ the subgroup of $\Aut(P^\vee)$ generated by the simple reflections $s_i$ for $i \in I$. The Weyl group is defined by $W = W_0 \rtimes P^\vee$, with the product $(s,x)(s',y) = (ss',s'^{-1}(x)+y)$. For $s \in W_0$, we write $s = (s,0)$ and for $x \in P^\vee$, we write $x = (1,x)$. Consider $\theta = \alpha_1 + \dots + \alpha_n$ the highest root of $\g$ and denote $s_0 = (s_{\theta},\theta^\vee)$. Consider the normal subgroup $\tilde{W}$ of $W$ generated by $(s_0,s_1,\dots,s_n)$ and the quotient $\mathcal{T} = W/\tilde{W}$. Then $W = \tilde{W} \rtimes \mathcal{T}$ and the action of $\mathcal{T}$ on $\tilde{W}$ is given by $\tau s_i \tau^{-1} = s_{\tau(i)}$ for $\tau \in \mathcal{T}$ and $i \in \hat{I}$. Finally, extend the length function of $\tilde{W}$ to $W$ by $l_W(\tau w) = l_{\tilde{W}}(w)$ for $\tau \in \mathcal{T}$ and $w \in \tilde{W}$.

Now, consider the braid group $\mathcal{B}$ generated by $\mathcal{T}_w$ for $w \in W$ with relations $\mathcal{T}_{ww'} = \mathcal T_w\mathcal T_{w'}$ if $l(ww') = l(w) + l(w')$. A reduced presentation of $w \in W$ is an expression $w = \tau s_{i_1}\dots s_{i_N}$ where $l(w) = n$ for $\tau \in \mathcal{T}$. For example, for $\g = \ssl_3$, in $W$:
\[ \omega_1^\vee = \tau_1 s_2s_1; \quad \omega_2^\vee = \tau_2 s_1s_2; \quad \mathcal T_{\omega_1^\vee} = \mathcal T_{s_0}\mathcal T_{s_2} \mathcal T_{\tau_1}; \quad \mathcal T_{\omega_2^\vee} = \mathcal T_{s_0} \mathcal T_{s_1} \mathcal T_{\tau_2}\]
where $\tau_1 = (012)$ and $\tau_2 = (021) = \tau_1^{-1}$.

Finally, recall the action of the braid group associated with $\tilde{W}$ on $U_q(\hat\g)$, denoting $\mathcal T_i := \mathcal T_{s_i}$ for $i \in I$.
\begin{gather*}
    \mathcal T_i(E_i) = -F_iK_i; \quad \mathcal T_i(E_j) = \sum\limits_{s=0}^{-c_{ij}}(-1)^{s-c_{ij}}q^{-s}E_i^{-c_{ij}-s}E_jE_i^s \quad \text{if } j\neq i;\\
    \mathcal T_i(F_i) = -K_i^{-1}E_i; \quad  \mathcal T_i(F_j)  = \sum\limits_{s=0}^{-c_{ij}}(-1)^{s-c_{ij}}q^{s}F_i^{s}F_jF_i^{-c_{ij} - s} \quad \text{if } j\neq i; \\
    \mathcal T_i(K_j) = K_{s_i(\alpha_j)}.
\end{gather*}
Moreover, this action is compatible with the automorphism $\Phi$ and the anti-automorphism $\Omega$ in the following sense:
\[ \mathcal T_i\Phi = \Phi^{-1} \mathcal T_i; \quad \mathcal T_i\Omega = \Omega \mathcal T_i.\]

Now, following \cite{Damiani}, we define $2\rho^\vee := \omega_2^\vee + \omega_1^\vee + \omega_2^\vee + \omega_1^\vee$. As an element of the Weyl group, one can then write:
\[ 2\rho^\vee = \omega_2^\vee \omega_1^\vee \omega_2^\vee \omega_1^\vee = s_0s_1s_2s_1s_0s_1s_2s_1. \]
One should keep in mind that the choice of the expression of $2\rho^\vee$ is important: the root vectors and the ordering of the real positive roots will depend on it, and so will the definition of the univ R-matrix.

We then define the sequence $(i_1,i_2,i_3,i_4,i_5,i_6,i_7,i_8) = (0,1,2,1,0,1,2,1)$ and the cyclic infinite sequence
\[ \begin{array}{ccccc}
    \iota & : & \Z & \to & \hat I  \\
     &  & k  &\mapsto & i_{k\text{ mod }8} 
\end{array} \]
to construct the real positive roots as in \cite[Proposition 3.2]{Damiani}. For $k \in \Z$, let
    \[ \beta_k := \left\{ \begin{array}{cc}
         s_{\iota(1)} \dots s_{\iota(k-1)} (\alpha_{\iota(k)}) & \text{if } k \geq 1   \\
         s_{\iota(0)} \dots s_{\iota(k+1)}(\alpha_{\iota(k)})& \text{if } k \leq 0
    \end{array} \right.\]
    Then $\{\beta_k ~|~k \leq 0\} = \hat R_+$ and $\{ \beta_{k} ~|~k\geq 1\}= \hat R_-$. For example, the first few terms for $\hat R_\pm$ are
\begin{gather*} \beta_1 = \delta - \alpha_1 - \alpha_2; \quad \beta_2 = \delta - \alpha_2; \quad \beta_3 = 2\delta-\alpha_1 - \alpha_2; \quad  \beta_4 = \delta-\alpha_1. \\
 \beta_0 = \alpha_1 ; \quad  \beta_{-1} = \alpha_1 + \alpha_2; \quad 
\beta_{-2} = \alpha_2.
\end{gather*}
We fix an ordering on $\hat{R}$ by
\[ \beta_1 \prec \beta_2 \prec \beta_3 \prec \beta_4 \dots \prec \beta_{-2} \prec \beta_{-1} \prec \beta_0. \]
We define the \textbf{root vectors} associated with the real positive roots $\beta_k$ from \cite[Definition 7]{Damiani} by
    \[ E_{\beta_k} := \left\{ \begin{array}{cc}
        \mathcal T_{\iota(1)} \dots \mathcal T_{\iota(k-1)}\po E_{\iota{(k)}} \pf & \text{if } k \geq 1  \\
        \mathcal T_{\iota(0)}^{-1} \dots \mathcal T_{\iota(k+1)}^{-1} \po  E_{\iota(k)}\pf & \text{if } k \leq 0 
    \end{array} \right.\]
    and $F_{\beta_k} = \Omega\po E_{\beta_k} \pf$. It is known that $E_{\beta_k} \in U_q(\mathfrak b)$ for $k \in \Z$.

    \begin{remark}
    The general formulas for the root vectors associated with $m\delta + \alpha_i$ and $r\delta -\alpha_i$ for $m \geq 0, r \geq 1, i\in I$ are given in \cite{Frenkandez}:
    \[ E_{m\delta +\alpha_i} = x_{i,m}^+; \quad F_{m\delta + \alpha_i} = x_{i,-m}^-;\\
    E_{r\delta-\alpha_i} = -K_i^{-1}x_{i,r}^-; \quad F_{r\delta-\alpha_i} = -x_{i,-r}^+K_i. \]
\end{remark}

    Let $U_q^+(\mathfrak b)$ be the subalgebra of $U_q(\mathfrak b)$ generated by the root vectors $E_{\beta_k}$ for $k \leq 0$. Then it has a basis 
    \[\{  E_{\beta_0}^{a_0} E_{\beta_{-1}}^{a_{-1}} E_{\beta_{-2}}^{a_{-2}} \dots ~ | ~  a_0,a_{-1},a_{-2} \dots \in \N \text{ such that almost all } a_k = 0  \}.\]
    Similarly, the other monomials in the roots vectors $E_{\beta_k}$ for $k > 0$ span another subalgebra $U_q^-(\mathfrak b)$ of $U_q(\mathfrak b)$. Let $U_q^0(\mathfrak b)$ be the subalgebra of $U_q(\mathfrak b)$ generated by $K_i^{\pm 1},h_{i,s}$ for $i \in I, s >0$. Then the multiplication defines a linear isomorphism $$U_q^+(\mathfrak b) \otimes U_q^0(\mathfrak b) \otimes U_q^-(\mathfrak b) \cong U_q(\mathfrak b). $$
     %The $(F_{\beta_k})$ define a similar PBW basis for $U_q(\mathfrak{c})$.

In the next sections, we will be interested in some particular root vectors computed as follows.

\begin{proposition}
    For $i \in I$, we have the following formulas in $U_q(\mathfrak b)$:
    \begin{gather*}
    E_{\alpha_i} = x_{i,0}^+; \quad F_{\alpha_i} = x_{i,0}^-;  \quad E_{\delta- \alpha_i} = -K_i^{-1}x^-_{i,1}; \quad F_{\delta-\alpha_i} = -x_{i,-1}^+K_i\\
    E_{\alpha_0} = E_0; \quad F_{\alpha_0} = F_0; \quad 
    E_{\alpha_1 + \alpha_2} = - \co x_{2,0}^+,x_{1,0}^+ \cf_{q^{-1}}; \quad F_{\alpha_1 + \alpha_2} =  -\co x_{1,0}^-,x_{2,0}^- \cf_q 
\end{gather*}
where $\co x,y \cf_q = xy - qyx$.
\end{proposition}

\begin{proof}
    The first two root vectors are $
    E_{\alpha_1} = E_1 ; \quad E_{\alpha_0} = E_0.$

We can also easily compute:
\[ E_{\alpha_1 + \alpha_2} = \mathcal T_1^{-1}(E_2) = \Phi \mathcal T_1(F_2) = \Phi(-F_2F_1 + qF_1F_2) = q^{-1}E_1E_2 - E_2E_1 = - \co x_{2,0}^+,x_{1,0}^+ \cf_{q^{-1}}.\]

By definition, there is no reason for $E_{\alpha_2} = E_2$. Let us check this properly.
   \begin{align*} 
    E_{\alpha_2} &= \mathcal T_1^{-1}\mathcal T_2^{-1}(E_1) = \Phi \mathcal T_1 \mathcal T_2(F_1) = \Phi \mathcal T_1\po -F_1F_2 + qF_2F_1 \pf \\ 
    &= \Phi \po K_1^{-1}E_1(-F_2F_1 + q F_1F_2) - q(-F_2F_1 + qF_1F_2)K_1^{-1}E_1 \pf
    \end{align*}
Since 
\begin{align*} K_1^{-1}E_1(-F_2F_1 + q F_1F_2) =& -qF_2F_1K_1^{-1}E_1 - q^{-1}F_2 \po \frac{1-K_1^{-2}}{q-q^{-1}} \pf + q^2F_1F_2K_1^{-1}E_1 \\ &+\frac{q}{q-q^{-1}} F_2 - q^{-1} F_2\frac{K_1^{-2}}{q-q^{-1}} 
\end{align*}
we finally have:
\[ E_{\alpha_2} = \Phi(F_2) = E_2 = x_{2,0}^+.\]

Now, we want to compute $E_{\delta - \alpha_1}$. We will do the same kind of computation as for $E_{\alpha_2}$, but we want an expression in Drinfeld's new realization. According to \cite[Definition 4.6]{Beck}:
\[ x_{1,1}^- = \mathcal T_{\omega_1^\vee}(F_1) = \mathcal T_0 \mathcal T_2(F_2) = \mathcal T_0(-K_2^{-1}E_2) = K_0^{-1}K_2^{-1} \po E_0E_2 - q^{-1}E_2E_0 \pf = K_1 \co E_0,E_2 \cf_{q^{-1}}.\]

On the other hand, one can check:
\[ E_{\delta-\alpha_1} = \mathcal T_0 \mathcal T_1 \mathcal T_2(E_1) = \mathcal T_0(E_2) = - \co E_0,E_2 \cf_{q^{-1}} = -K_1^{-1}x_{1,1}^-.\]

To compute the root vectors associated with the lower Borel subalgebra, we apply $\Omega$ to these results:
\begin{gather*} F_{\alpha_1} = F_1; \quad F_{\alpha_0} = F_0; \\
F_{\alpha_1 + \alpha_2} = qF_2F_1 - F_1F_2 = -\co x_{1,0}^-,x_{2,0}^- \cf_{q^{-1}} \\
F_{\alpha_2} = F_2.
\end{gather*}

For $F_{\delta-\alpha_1}$, we need to compute:
\[ \Omega(x_{1,1}^-) = \Omega\po K_1 \po -E_0E_2 + q^{-1}E_2E_0 \pf\pf = F_2F_0K_1^{-1} - qF_0F_2K_1^{-1}.\]

Since, by definition 4.6 in \cite{Beck}, 
\[ x_{1,-1}^+ = \mathcal T_{\omega_1^\vee}(E_1) = \mathcal T_0 \mathcal T_2(E_2) = \mathcal T_0(-F_2K_2) = \co F_2,F_0 \cf_{q} K_1^{-1}\]
hence
\[ \Omega(x_{1,1}^-) = x_{1,-1}^+.\]
Therefore:
\[ F_{\delta-\alpha_1} = \Omega(-K_1^{-1}x_{1,1}^-) = -x_{1,-1}^+K_1.\]
\end{proof}

To conclude this section, let us define the triangular parts of the universal R-matrix. We recall the definition of the q-exponential for $x \in U_q(\hat\g)$:
\[ \exp_q(x) = \sum\limits_{k \geq 0} \frac{x^k}{(k)_q!}. \]

Consider the following completion of $U_q(\mathfrak b) \otimes U_q(\mathfrak c)$.  $$U_q(\mathfrak b) \tilde\otimes U_q(\mathfrak c) := \sum\limits_{\alpha,\beta \in Q}\po \prod\limits_{\gamma \in Q_+} U_q(\mathfrak b)_{\alpha + \gamma} \otimes U_q(\mathfrak c)_{\beta-\gamma} \pf\subset \prod\limits_{\alpha,\beta \in Q} U_q(\mathfrak{b})_\alpha \otimes U_q(\mathfrak{c})_\beta.$$ It is an algebra.

\begin{definition}\cite[(11.2)]{Damiani}\cite[(8.3)]{Zhang}
    The \textbf{triangular parts} of the universal R-matrix are the power series in $z$ with coefficients in  $U_q(\mathfrak{b}) \tilde\otimes U_q(\mathfrak{c})$:
    \begin{align*} \mathcal{R}^+(z) = &\prod\limits_{m \geqslant 0, ~ \alpha \in R} \exp_q\po (q^{-1}-q) E_{m\delta + \alpha} \otimes F_{m\delta + \alpha} z^m \pf;
    \end{align*}
    \begin{align*}
        \mathcal{R}^-(z) = &\prod\limits_{r > 0; ~ \alpha \in R} \exp_q\po (q^{-1}-q) E_{r\delta - \alpha} \otimes F_{r\delta - \alpha} z^r \pf
    \end{align*}
    where the factors are ordered in a decreasing manner according to the ordering fixed above.
\end{definition}

Now, let us define the $U_q(\mathfrak{b})$-module on which we will apply the first factor of these R-matrices.

\section{Positive prefundamental module}
We recall the results on lowest $\ell$-weight modules over $U_q(\mathfrak b)$, based on \cite{HernandezJimbo}. We are interested in the family of positive prefundamental modules, which are representations of $U_q(\mathfrak b)$ that are not representations of the whole quantum affine algebra $U_q(\hat \ssl_3)$. Here, we determine the actions of all the root vectors on a positive prefundamental representation.

Let us first recall the weight structure of the representations of $U_q(\mathfrak{b})$. For $V$ a $U_q(\mathfrak{b})$-module, for $\lambda \in (\C^\times)^I$, consider:
\[ V_\lambda := \{ v \in V~ | ~ \phi_{i,0}^+ \cdot v = \lambda(i)v \}. \]
If $V_\lambda \neq 0$, it is called the weight space of weight $\lambda$. A module $V$ is called weight graded if it is a direct sum of its weight spaces. Moreover, observe that the root lattice $Q$ can be embedded into $\mathfrak t^\times$ by $\alpha_j \mapsto \overline{\alpha_j} = \po q^{c_{ij}} \pf_{i,j \in I}$. This application is injective, as $q$ is not a root of unity. Denote by $\overline Q$ the image of $Q$. If all the weights of a module belong to $ \overline Q$, then the module is called \textbf{root graded}. For $v \in V_\lambda$, denote $\wt(v) := \lambda$.

The set of $\ell$-weight $\mathfrak{t}_l := \po \C[[z]]^\times \pf^I$ is the set of $I$-tuples of invertible power series in $z$. For $j \in I$ and $a \in \C$, we define the \textbf{positive prefundamental} $\ell$-weight $\Psi_{j,a} := \po (1-az)^{\delta_{ij}} \pf_{i\in I}$.

For $\mathbf{P} = (P_i(z))_{i\in I} \in \mathfrak t_\ell^\times$, the lowest $\ell$-weight Verma module $M(\mathbf{P})$ is the $U_q(\mathfrak{b}$)-module generated by a vector $v_0$ such that 
\[ \phi_i^+(z) \cdot v_0 = P_i(z)\ v_0; \quad E_{\alpha} \cdot v_0 = 0 \text{ for } \alpha \in \hat R_-. \]
By the triangular decomposition, we have $$M(\mathbf{P}) = U_q^+(\mathfrak b) \cdot v_0.$$
We denote by $L\po \mathbf{P} \pf$ its irreducible quotient. If the $\ell$-weight is positive prefundamental, we call this irreducible module \textbf{positive prefundamental module}.

In order to compute the monodromy matrices, we are interested in the particular case of positive prefundamental modules. For $\ssl_3$, we will only consider the $\ell$-weight $\Psi_{1,1} = \po 1-z,1 \pf$, since we will deduce the coproduct formula for $T_2$ from that for $T_1$. Denote by $L_1$ the positive prefundamental module associated with $\Psi_{1,1}$. Then, it is generated by a vector $v_{0,0}$ such that:
    \begin{gather*} \phi_1^\pm(z) \cdot v_{0,0} = (1-z)v_{0,0}; \quad \phi^\pm_2(z) \cdot v_{0,0} = v_{0,0}\\
    \quad E_{r\delta - \alpha_1} \cdot v_{0,0} = 0 \quad E_{r\delta - \alpha_2} \cdot v_{0,0} = 0 \quad E_{r\delta - \alpha_1- \alpha_2} \cdot v_{0,0} = 0 \text{ for } r\geq 1. 
    \end{gather*}  

However, Frankel-Hernandez \cite{Frenkandez} proved that most root vectors have a trivial action on this module. The actions of the remaining vectors are explicitly computable.

\begin{lemme}\cite{HernandezJimbo}\cite[Theorem 6.3]{Frenkandez}
    \begin{enumerate}
        \item The following generators act by $0$ on $L_1$, for $j \in I, r \geqslant 1$:
    \[x_{j,r}^+; \quad E_{(r+1)\delta - \alpha_1}; \quad E_{r\delta - \alpha_2}; \quad E_{(r+1)\delta - (\alpha_1+\alpha_2)}; \quad \phi_{j,\delta_{i,j}+r}^\pm.\]
        \item We have the weight decomposition:
        \[ L_1 = \bigoplus\limits_{a,b\in\N}\po L_1\pf_{\overline{a\alpha_1}} \oplus \po L_1 \pf_{\overline{b(\alpha_1 +\alpha_2)}}\]
        where each weight space is of dimension $1$.
        \item From the previous point, we have $E_2\cdot v_{0,0} = 0$.
    \end{enumerate}    
\end{lemme}

We now provide an explicit presentation of this positive prefundamental module. 

\begin{proposition}\label{moduleprefondamental}
    The module $L_1$ is the $U_q(\mathfrak{b})$-module with basis $(v_{a,b})_{a,b\in \Z_+}$ and actions, for $m \geqslant 0$ and $r \geqslant 1$:
\begin{gather*}
    x_{1,m}^+ \cdot v_{a,b} = \delta_{m=0} v_{a+1,b}; \quad x_{1,r}^-\cdot v_{a,b} = \delta_{r=1} q^{b}\frac{(a)_q}{q-q^{-1}}v_{a-1,b}; \\
    x^+_{2,m} \cdot v_{a,b} = -\delta_{m=0}q^{-a}(a)_{q} v_{a-1,b+1}; \quad x_{2,r}^- \cdot v_{a,b}= 0; \\
    E_{\alpha_1 + \alpha_2} \cdot v_{a,b} = q^{a}v_{a,b+1}; \quad 
    E_0 \cdot v_{a,b} = -q^{-a-2(b-1)}\frac{(b)_q}{q-q^{-1}}  v_{a,b-1}; \\
    K_1 \cdot v_{a,b} = -\phi_{1,1}^\pm\cdot v_{a,b}  = q^{2a+b}v_{a,b}; \quad K_2 \cdot v_{a,b} = q^{b-a} v_{a,b}; \quad K_0 \cdot v_{a,b} = q^{-2b-a} v_{a,b}
\end{gather*}
and all the other root vectors and $(\phi_{j,\delta_{1,j}+r}^\pm)_{j \in I,r \geq1}$ act by $0$.
\end{proposition}

\begin{proof}
    Consider the basis of $U_q^+(\mathfrak b)$ of root vectors given in section 4. We have that $L_1$ is spanned by the following set $$\{ E_{\beta_0}^{a_0}E_{\beta_{-1}}^{a_{-1}} E_{\beta_{-2}}^{a_{-2}} \cdots  v_{0,0} \quad | \quad  a_0,a_{-1},a_{-2},\dots  \in \N, \text{ almost all } a_k = 0. \}$$ Moreover, according to the previous lemma, the following vectors span $L_1$:
    \[ v_{a,b} := E_{1}^{a}E_{\alpha_1 + \alpha_2}^b\cdot v_{0,0} \text{ for } a,b\in \N.\]
    Since each $v_{a,b}$ is of weight $\overline{a\alpha_1} + \overline{b(\alpha_1 + \alpha_2)}$, we deduce that all $v_{a,b}$ are non zero vectors of different weights that span $L_1$, thus forming a basis of $L_1$.

First, observe that, by definition, $x_{1,0}^+ \cdot v_{a,b} = v_{a+1,b}$. Using the quantum Serre's relations, we have $E_{\alpha_1 + \alpha_2}E_1 = qE_1E_{\alpha_1+\alpha_2}$, from which we deduce $E_{\alpha_1 + \alpha_2}\cdot v_{a,b} = q^av_{a,b+1}$.

    To compute the action of $x_{2,0}^+$, we take advantage of the expression of $E_{\alpha_1+\alpha_2}$. Indeed:
    \[ E_2E_1 = q^{-1}E_1E_2 - E_{\alpha_1+\alpha_2}.\]
    In addition, from the quantum Serre's relations, we know $E_2E_{\alpha_1+\alpha_2} = qE_{\alpha_1+\alpha_2}E_2$. Denote $E_2\cdot v_{a,b} = \lambda_{a,b}v_{a-1,b+1}$. Then:
    $ \lambda_{a,b} = q^{-1}\lambda_{a-1,b} - q^{a-1}$.
    Since $\lambda_{0,b} = 0$ (we already saw by grading arguments that $E_2 \cdot v_{0,0} = 0$), we finally have 
    $\lambda_{a,b} = -q^{-a}(a)_qv_{a-1,b+1}$.

    Now, denote $x_{1,1}^- \cdot v_{a,b} = \lambda_{a,b} v_{a-1,b}$. Knowing that $x_{1,1}^- E_1 = E_1x_{1,1}^- - \frac{\phi_{1,1}^+}{q-q^{-1}}$, we deduce \[ \lambda_{a,b} = \frac{q^{2(a-1)+b}}{q-q^{-1}} + \lambda_{a-1,b}.\]
    Moreover, one can check that $x_{1,1}^- E_{\alpha_1 + \alpha_2} = -E_2x_{1,1}^-E_1 + q^{-1}x_{1,1}^-E_1E_2$ and then deduce $\lambda_{0,b}= 0$. The above equation then has a unique solution given by: 
    \[ \lambda_{a,b} = q^b \frac{(a)_q}{q-q^{-1}}.\]

    Finally, denote $E_0\cdot v_{a,b} = \lambda_{a,b} v_{a,b-1}$. One can check $E_0E_1 = q^{-1} E_1E_0 - K_2^{-1}x_{2,1}^-$.
    Since $x_{2,1}^-$ acts by $0$ over $L_1$, we deduce $ \lambda_{a,b} = q^{-a}\lambda_{0,b}$.
    Then, using the quantum Serre's relations for $E_0$ and $E_2$:
    \[ q^{-(a-2)}\lambda_{0,b+2} -(q+q^{-1})q^{-(a-1)}\lambda_{0,b+1} + q^{-a}\lambda_{0,b} = 0.\]
    By definition, we know $\lambda_{0,0} = 0$. We just need to determine $\lambda_{0,1}$. Since $E_0E_2 = q^{-1}E_2E_0 + K_1^{-1}x_{1,1}^-$, we have:
    \[ E_0E_{\alpha_1 + \alpha_2} = q^{-2}E_{\alpha_1+\alpha_2}E_0 - K_1^{-1}x_{1,1}^-E_1 + q^{-2}E_1K_1^{-1}x_{1,1}^- - q^{-1}K_2^{-1}x_{2,1}^-E_2 + q^{-1}E_2K_2^{-1}x_{2,1}^-\]
    therefore $\lambda_{0,1} = -\frac{1}{q-q^{-1}}$. Finally, we have
    \[ \lambda_{a,b} = -q^{-a-2(b-1)} \frac{(b)_q}{q-q^{-1}}.\]
\end{proof}
%Note that $L_1$ can be equipped \cite{Frenkandez} with a $\N$-grading such that $(L_1)_0 := \C v_{0,0}$ and, for $m \geq 0, p\geq 0,  s > 0, j \in I, \beta \in R$:
%\begin{gather*} (L_1)_m \text{ is stable by } k_j^{\pm 1}, \quad E_{s\delta - \beta} \cdot \po L_1 \pf_m \subset \po L_1 \pf_{m-s}; \\ \quad \po h_{j,s} - \lambda_{j,s} \pf \cdot \po L_1 \pf_m \subset \po L_1 \pf_{m-s}; \quad E_{t\delta + \beta} \cdot \po L_1 \pf_m \subset  \sum_{p = 0}^{\langle \omega_1^\vee, \beta \rangle}\po L_1 \pf_{m-t+p}
%\end{gather*}
%where $h_{j,s} \cdot v_{0,0} = \lambda_{j,s} v_{0,0}$. Setting $\deg(v_{a,b}) = a+2b$, we obtain this $\N$-grading. In addition, note that their weights are given by $\wt(v_{a,b}) = a\alpha_1 + b(\alpha_1 +\alpha_2)$.

Now, we can define the monodromy matrices associated with $L_1$. These are defined by applying the first factor of the universal R-matrices on representations of the upper Borel subalgebra \cite[(4.32)]{Frenkikhin92}\cite[Definition 8.11]{Zhang}.

     Let $M$ be a $U_q(\mathfrak{b})$-module of lowest $\ell$-weight. Let $b_0 \in M$ be a lowest $\ell$-weight vector. Extend it to a weight basis $\mathbf{B}$. The \textbf{monodromy matrices} $(t^\pm_{b_1,b_2}(z))_{b_1,b_2 \in \mathbf B}$ associated with $M$ are the formal series with coefficients in $U_q(\mathfrak c)$ defined by
    \[ \mathcal{R}^\pm(z) \cdot\po b_2 \otimes 1 \pf = \sum\limits_{b_1 \in \mathbf B} b_1 \otimes t^\pm_{b_1,b_2}(z) \text{ for } b_2 \in \mathcal{ B}.\] 

With this construction, Zhang proved \cite[(9.52)]{Zhang} that Theta series can be expressed in terms of the entries of monodromy matrices. For $M = L_i$, we have, for example, for $\beta \in Q_+$:
\begin{equation}\label{thetamonodromy}
    \Theta_{i,\beta}(z) = \sum_{b \in \mathbf{B} : ~ \wt(b) = \overline{\beta}} t_{b,b_0}^+ (z)\otimes t_{b_0,b}^- (z)K_{\beta}^{-1}
\end{equation}
    where $K_\beta := K_1^{a_1}\dots K_{n}^{a_n}$ for $\beta = \sum\limits_{j\in I} a_j\alpha_j$.

We will apply this result with the basis $\mathbf{B} = (v_{a,b})_{a,b\in \N}$. Note that $\wt(v_{a,b}) = \overline{a\alpha_1} + \overline{b(\alpha_1 +\alpha_2)}$.

\section{Formula for the coproduct of T-series in type $A_2$}
We establish the second main result of this paper: the formula for the coproduct of T-series in type $A_2$. We compute the monodromy matrices associated with $L_1$ and deduce the formula for the Theta series. We then apply a morphism that exchanges the simple roots in $\ssl_3$ to obtain the formula for $\Delta(T_2(z))$. We conclude by checking the compatibility between our formula and the Finkelberg-Tsymbaliuk coproduct formulas
\cite{Finkeliuk}.

\begin{theoreme}\label{coproductTseries}
    The coproducts of the T-series for the quantum affine algebra $U_q(\hat\ssl_3)$ can be written, for $i \in \{ 1,2\}$
    \[ T_i(z) = \po 1 \otimes T_i(z) \pf \times \Theta_i(z) \times \po T_i(z) \otimes 1 \pf \]
    where the Theta series $\Theta_i(z)$ are given by the following two formulas:
    \begin{gather*}
    \Theta_1(z) = \exp_q\po (q-q^{-1})x_{1,0}^-\otimes x_{1,-1}^+z \pf \exp_q\po (q-q^{-1}) \co x_{1,0}^-,x_{2,0}^- \cf_q \otimes \co x_{2,0}^+,x_{1,{-1}}^+ \cf_{q^{-1}} z \pf ; \\
    \Theta_2(z) =  \exp_q\po (q-q^{-1}) x_{2,0}^-\otimes x_{2,-1}^+ z \pf  \exp_q\po (q-q^{-1}) \co x_{2,0}^-,x_{1,0}^- \cf_q \otimes \co x_{1,0}^+,x_{2,-1}^+  \cf_{q^{-1}} z \pf.
    \end{gather*}
    Moreover, the q-exponentials commute in both formulas.
\end{theoreme}

\begin{proof}
First, we prove that one can deduce the formula for $\Theta_2(z)$ from the formula for $\Theta_1(z)$. Indeed, consider the Hopf algebra automorphism of $U_q\po \hat{\ssl_3} \pf$ :
\[ \psi : \left|  \begin{array}{cc}
     E_1 \mapsto E_2; & E_2 \mapsto E_1;  \\
     F_1 \mapsto F_2; & F_2 \mapsto F_1; \\
    K_1 \mapsto K_2; & K_2 \mapsto K_1.
\end{array} \right.  \]
For simplicity, we also denote $\psi(0) = 0, \psi(1) = 2, \psi(2) = 1$. One can check $\psi \mathcal T_i = \mathcal T_{\psi(i)} \psi$ for $i \in \hat I$. Then, using definition 4.6 in \cite{Beck}, one has for $s > 0$
\[ \psi(x_{1,-s}^+) = (-1)^s x_{2,-s}^+.\]
Since $ \co  x_{i,-s}^+,x_{i,0}^-,\cf = \frac{\phi_{i,-s}^-}{q^{-1}-q}$, we deduce $\psi(\phi_{1,-s}^-) = (-1)^s\phi^-_{2,-s}$. Then, by induction on $s$, one has $\psi(h_{1,-s}) = (-1)^sh_{2,-s}$, so we have
$\psi \po T_1(z) \pf = T_2(-z)$ and the formula for $T_2(z)$ from $T_1(z)$.

Now, let us prove the formula for $\Theta_1(z)$. Consider the first positive prefundamental module $L_1$ and $(v_{a,b})_{a,b\in \N}$ its basis described in proposition \ref{moduleprefondamental}. To obtain the positive and negative parts of the monodromy matrix associated with $L_1$, we have to compute the following:
\begin{gather*} \mathcal{R}^+(z) \cdot (v_{0,0} \otimes 1) = \sum\limits_{a,b \in \mathbb N} v_{a,b} \otimes t^+_{v_{a,b},v_{0,0}}(z); \\
\mathcal R^-(z) \cdot (v_{a,b}\otimes 1) = \sum\limits_{k,l\in\mathbb N} v_{k,l} \otimes t^-_{v_{k,l},v_{a,b}}(z)
\end{gather*}
Note that we are only interested in $t^-_{v_{0,0},v_{a,b}}(z)$ for the negative part.

Let us start with the positive part. Since we know that almost all terms of $\mathcal{R}^+(z)$ vanish over $M$ according to the proposition \ref{moduleprefondamental}, we have that the term $ \mathcal{R}^+(z) \cdot (v_{0,0} \otimes 1)$ is equal to
\[\sum\limits_{a,b,c \in \N}\frac{\po (q^{-1}-q) E_1 \otimes F_1 \pf^a}{(a)_q!} \frac{\po (q^{-1}-q)E_{\alpha_1 + \alpha_2} \otimes F_{\alpha_1 + \alpha_2} \pf^b}{(b)_q!} \frac{\po (q^{-1}-q) E_2 \otimes F_2 \pf^c}{(c)_q!} \cdot (v_{0,0}\otimes 1)\]

But with the same proposition:
\[
    E_2^c \cdot v_{0,0} = \delta_{c=0} v_{0,0}; \quad
    E_{\alpha_1 + \alpha_2}^b = v_{0,b}; \quad
    E_1^a \cdot v_{0,b} = v_{a,b}.
\]
Hence:
\[ \mathcal{R}^+(z) \cdot (v_{0,0} \otimes 1) = \sum\limits_{a,b \in \N} v_{a,b} \otimes \frac{(q^{-1}-q)^a}{(a)_q!}F_1^a \frac{(q^{-1}-q)^b}{(b)_q!}F_{\alpha_1 + \alpha_2}^b. \]
Finally, we have
%\[ t^+_{v_{a,b},v_{0,0}}(z) = \frac{(q^{-1}-q)^a}{(a)_q!}F_1^a \frac{(q^{-1}-q)^b}{(b)_q!}F_{\alpha_1 + \alpha_2}^b \]
%which we can rewrite in terms of generators of the Drinfeld's new realization:
\[ t^+_{v_{a,b},v_{0,0}}(z) = \frac{(q^{-1}-q)^a}{(a)_q!}(x_{1,0}^-)^a \frac{(q-q^{-1} )^b}{(b)_q!} \co x_{1,0}^-,x_{2,0}^- \cf_q^b\]

Now, let us compute the negative part. Again, we know that almost all terms of $\mathcal{R}^-(z)$ vanish over $M$, then:
\[ \mathcal{R}^-(z) \cdot (v_{a,b} \otimes 1) = \sum\limits_{k,l\in \mathbb N} \frac{\po (q-q^{-1}) E_{\delta-\alpha_1}\otimes F_{\delta-\alpha_1} z \pf^k}{(k)_q!} \frac{\po (q-q^{-1}) E_0\otimes F_0 z \pf^l}{(l)_q!} \cdot (v_{a,b}\otimes 1). \]

Yet, since:
\[ E_0^l \cdot v_{a,b} \in \mathbb Cv_{a,b-l} \text{ and } E_{\delta-\alpha_1}^k \cdot v_{a,b-l} \in \mathbb Cv_{a-k,b-l}\]
we are only interested in the terms for $k=a$ and $l=b$:
\[ E_0^b \cdot v_{a,b} = (-1)^b \frac{(b)_q!}{(q-q^{-1})^b}q^{-ab-b(b-1)}v_{a,0}; \quad E_{\delta - \alpha_1}^a \cdot v_{a,0} = (-1)^a \frac{(a)_q!}{(q-q^{-1})^a}q^{-a(a-1)}.\]

Finally:
\[ t^-_{v_{0,0},v_{a,b}}(z) = q^{-a(a-1)}F_{\delta-\alpha_1}^a z^a  q^{-ab-b(b-1)}F_ 0^bz^b.\]

To conclude, we just have to rewrite this expression with the generators of the Drinfeld new realization:
\begin{gather*} F_{\delta - \alpha_1}^a = (-x_{1,-1}^+K_1)^a = (-1)^aq^{a(a-1)}(x_{1,-1}^+)^aK_1^a; \\
K_1^aF_0^b = q^{ab}F_0^bK_1^a;\\
F_0^b = \po \co x_{2,0}^+, x_{1,-1}^+ \cf_q^{-1} K_0^{-1} \pf^b = q^{b(b-1)} \po \co x_{2,0}^+,x_{1,-1}^+\cf_{q^{-1}} \pf^b (K_0^{-1})^b.
\end{gather*}

Hence:

\[ t^-_{v_{0,0},v_{a,b}}(z) = (-1)^a \po x_{1,-1}^+ \pf^a \po \co x_{2,0}^+,x_{1,-1}^+ \cf_{q^{-1}} \pf ^b K_1^a (K_0^{-1})^bz^{a+b}. \]

Using equation \eqref{thetamonodromy}, we have:
\[ \Theta_{1,a\alpha_1 + b(\alpha_1 + \alpha_2) + c\alpha_2} = \delta_{c=0}\frac{(q-q^{-1})^a}{(a)_q!} (x_{1,0}^-)^a\frac{(q-q^{-1} )^b}{(b)_q!} \co x_{1,0}^-,x_{2,0}^- \cf_q^b\otimes (x_{1,-1}^+)^a \co x_{2,0}^+,x_{1,-1}^+ \cf_{q^{-1}} ^b z^{a+b} \]
so
\[ \Theta_1(z) = \exp_q\po (q-q^{-1})x_{1,0}^-\otimes x_{1,-1}^+z \pf \exp_q\po (q-q^{-1}) \co x_{1,0}^-,x_{2,0}^- \cf_q \otimes \co x_{2,0}^+,x_{1,{-1}}^+ \cf_{q^{-1}} z \pf.\]

Moreover, the terms in the q-exponential commute. For example, for $\Theta_1(z)$, note that
\[ x_{1,0}^-\co x_{1,0}^-,x_{2,0}^- \cf_q = q^{-1} \co x_{1,0}^-,x_{2,0}^- \cf_qx_{1,0}^-; \quad x_{1,-1}^+ \co x_{2,0}^+,x_{1,-1}^+\cf_{q^{-1}} = q \co x_{2,0}^+,x_{1,-1}^+\cf_{q^{-1}} x_{1,-1}^+. \]
Thus, the q-exponentials commute. 
\end{proof}

    One can check that these formulas are compatible with the formulas of the coproduct of the Drinfeld-Cartan elements given by Finkelberg-Tsymbaliuk in \cite[Theorem 10.14]{Finkeliuk}.

    In type $A_2$, the inverse of the quantum Cartan matrix is $\tilde{C}(q) = \frac{1}{[3]_q} \begin{pmatrix}
       [2]_q  &  1 \\
        1 & [2]_q
    \end{pmatrix}$. Now, we look at the coefficient of $z$ in $T_1(z)$ (resp. $T_2(z)$) that we denote $T_{1,1}$ (resp. $T_{2,1}$). Then we have $$T_{1,1} = \frac{[2]_q}{[3]_q}h_{1,-1} + \frac{1}{[3]_q}h_{2,-1}; \quad T_{2,2}= \frac{1}{[3]_q} h_{1,-1} + \frac{[2]_q}{[3]_q} h_{2,-1}.$$ Note that $[2]_q^2 - 1 = [3]_q$, then we get \[ \left\{ \begin{array}{cc}
         h_{1,-1} & = [2]_q T_{1,1} - T_{2,1}  \\
         h_{2,-1} & = [2]_q T_{2,1} - T_{1,1}
    \end{array} \right..\]

    On the other hand, looking at the coefficient $z$ in the theorem \ref{coproductTseries}:
    \begin{align*} \Delta\po  T_{1,1} \pf =& \frac{[2]_q}{[3]_q}  \square(h_{1,-1})  + \frac{1}{[3]_q}\square\po h_{2,-1} \pf + (q-q^{-1})x_{1,0}^- \otimes x_{1,-1}^+ \\ & + (q-q^{-1}) \co x_{1,0}^-,x_{2,0}^-\cf_q \otimes \co x_{2,0}^+, x_{1,-1}^+\cf_{q^{-1}}  \\
    \Delta\po  T_{2,1} \pf =& \frac{[2]_q}{[3]_q}  \square(h_{2,-1})  + \frac{1}{[3]_q}\square\po h_{1,-1} \pf + (q-q^{-1})x_{2,0}^- \otimes x_{2,-1}^+ \\ &+ (q-q^{-1}) \co x_{2,0}^-,x_{1,0}^-\cf_q \otimes \co x_{1,0}^+, x_{2,-1}^+\cf_{q^{-1}} 
    \end{align*}
    %\[ \Delta\po  \frac{[2]_q}{[3]_q}h_{1,-1} + \frac{1}{[3]_q} h_{2,-1} \pf = \frac{[2]_q}{[3]_q}  \square(h_{1,-1})  + \frac{1}{[3]_q}\square\po h_{2,-1} \pf + (q-q^{-1})x_{1,0}^- \otimes x_{1,-1}^+ + (q-q^{-1}) \co x_{1,0}^-,x_{2,0}^-\cf_q \otimes \co x_{2,0}^+, x_{1,-1}^+\cf_{q^{-1}}  \]
    where $\square(x) = x\otimes 1 + 1 \otimes x$. Note that \[ \co x_{1,0}^+,x_{2,-1}^+ \cf_{q^{-1}} = - \co x_{2,0}^+, x_{1,-1}^+\cf_{q^{-1}}; \quad [2]_q \co x_{1,0}^-,x_{2,0}^-  \cf_q + \co x_{2,0}^-,x_{1,0}^- \cf_q = -q^{2} \co x_{2,0}^-, x_{1,0}^- \cf_{q^{-3}} \] Then, one can compute:
    \begin{align*}
    \Delta(h_{1,-1}) =& \Delta \po [2]_q T_{1,1} - T_{2,1}\pf \\  =&\square(h_{1,-1}) + (q^2-q^{-2}) x_{1,0}^-\otimes x_{1,-1}^+ - (q-q^{-1})x_{2,0}^-\otimes x_{2,-1}^+  \\&- (q^3-q) \co x_{2,0}^-,x_{1,0}^- \cf_{q^{-3}} \otimes \co x_{2,0}^+,x_{1,-1}^+ \cf_{q^{-1}} \\
        \Delta(h_{2,-1}) =& \Delta\po [2]_qT_{2,1} - T_{2,1} \pf \\ =&\square(h_{2,-1}) +  \po q^2 - q^{-2} \pf x_{2,0}^- \otimes  x_{2,-1}^+ - (q-q^{-1}) x_{1,0}^- \otimes  x_{1,-1}^+ \\&+ (q^3 -q) \co x_{1,0}^-, x_{2,0}^-  \cf_{q^{-3}} \otimes \co x_{2,0}^+,x_{1,-1}^+ \cf_{q^{-1}}.
    \end{align*} 
    which are the formulas obtained in \cite{Finkeliuk}. 
    
    %Using the formulas for the coproducts of \cite{Finkeliuk}, recalled here at the end of section 3.1, and noting that
    %\[ \frac{[2]_q}{[3]_q}\po q^2 - q^{-2} \pf - \frac{1}{[3]_q}\po q-q^{-1} \pf = \frac{\po q- q^{-1}  \pf}{[3]_q} \po \po q+ q^{-1} \pf^2 - 1 \pf = \po q-q^{-1} \pf \]
    %and
    %\[ -q^2[2]_q\po q-q^{-1} \pf \co x_{2,0}^-,x_{1,0}^- \cf_{q^{-3}} + q^2\po q - q^{-1} \pf \co x_{1,0}^-,x_{2,0}^- \cf_{q^{-3}} = \po q^{-2} - q^4 \pf x_{2,0}^-x_{1,0}^- + \po q^3 - q^{-3} \pf x_{1,0}^-x_{2,0}^- \]
    %we deduce
    %\[ -q^2[2]_q\po q-q^{-1} \pf \co x_{2,0}^-,x_{1,0}^- \cf_{q^{-3}} + q^2\po q - q^{-1} \pf \co x_{1,0}^-,x_{2,0}^- \cf_{q^{-3}} =  [3]_q\po q - q^{-1} \pf \co x_{1,0}^-,x_{2,0}^- \cf_q\]
    %hence the correspondence.

\section{Appendix : shifted Yangians and Theta series as associators}
In this section, we will prove the proposition \ref{intertwiningyangian}, following the proof of \cite[Theorem 5.7]{Zhang}. We shall need the notion of shifted Yangians, a family of algebras that generalize the standard Yangian. It was introduced in by Brundan-Kleshchev \cite{BrundanKleshchev} in type A and \cite{GKLOtermes, BFN} in full generality. Theta series can be interpreted as associators for particular representations of shifted Yangians. From this property, we check that the Theta series must satisfy the equations of proposition \ref{intertwiningyangian}.

\begin{definition}
    Let $\mu = \sum\limits_{i\in I} m_i \omega_i^\vee \in P^\vee$ a coweight. The \textbf{shifted Yangian} $Y_\mu$ is the associative algebra defined by generators $x_{i,m}^-,x_{i,m}^+,\xi_{i,p}$ for $i \in I, m \in \mathbb N, p \in \mathbb Z$ that satisfy the following relations:
    \begin{gather*}
        \co \xi_{i,p},\xi_{j,p'}\cf = 0, \quad \co x^+_{i,m},x^-_{j,m'} \cf = \delta_{i,j}\xi_{i,m+m'}, \\ 
        \co \xi_{i,p+1},x^\pm_{j,m} \cf - \co \xi_{i,p},x^\pm_{j,m+1} \cf = \pm \frac{1}{2}c_{ij}\po \xi_{i,p} x^\pm_{j,m} + x^\pm_{j,m}\xi_{i,p} \pf,  \\
        \co x^\pm_{i,m+1},x^\pm_{j,m'} \cf - \co x^\pm_{i,m},x^\pm_{j,m'+1} \cf = \pm \frac{1}{2}c_{ij} \po x^\pm_{i,m} x^\pm_{j,m'} + x^\pm_{j,m'} x^\pm_{i,m} \pf, \\
        \ad^{1-c_{ij}}_{x_{i,0}^\pm}(x_{j,0}^\pm) = 0 \quad \quad \text{if }i\neq j, \\
        \xi_{i,-m_i-1} = 1; \quad \xi_{i,p} = 0 \text{ for } p < -m_i -1.
    \end{gather*}
    Define the generating series
    \[ \xi_i(z) := \sum\limits_{p\in \Z} \xi_{i,p}z^{-p-1}; \quad x_{i}^\pm(z) := \sum\limits_{m\in\N} x_{i,m}^\pm z^{-m-1}\]
    which are Laurent series with coefficients in $Y_\mu$ and leading terms $z^{m_i}$ and $x_{i,0}^\pm z^{-1}$.
    
    A coweight $\sum\limits_{i\in I} m_i \omega_i^\vee \in P^\vee$ is said to be dominant (resp. antidominant) if $m_i \geq 0$ (resp. $m_i \leq 0$) for all $i \in I$.
\end{definition}

A generalization of the standard coproduct was introduced in \cite{CoproductYangian}, called shifted coproducts. 

\begin{theoreme}\cite[Corollary 3.16, Theorem 4.12, Proposition 4.14]{CoproductYangian}
    \begin{enumerate}
        \item For $\epsilon, \eta$ antidominant coweights, the shift homomorphisms $\iota^\mu_{\epsilon,\eta} : Y_{\mu}  \longrightarrow  Y_{\mu + \epsilon + \eta}$ defined by: 
        \[ x^+_{i,n} \mapsto  x^+_{i,n-\langle\epsilon,\alpha_i\rangle}; \quad 
              x^-_{i,n} \mapsto  x^-_{i,n-\langle\eta,\alpha_i\rangle}; \quad
               \xi_{i,p} \mapsto  \xi_{i,p - \langle \epsilon +\eta,\alpha_i\rangle} \]
        are injective.
        \item There exists a unique family of algebra homomorphisms
        \[ \Delta_{\mu,\nu} : Y_{\mu + \nu} \longrightarrow Y_\mu\otimes Y_\nu \]
        for all $\mu,\nu \in \Z$ such that $\Delta_{0,0}$ is the coproduct of the standard Yangian, 
        \[ \Delta_{\mu,\nu}(x^+_{i,n}) = x^+_n \otimes 1 \text{ if } n < -\langle\mu,\alpha_i\rangle, \quad \Delta_{\mu,\nu}(x^-_{i,n}) = 1 \otimes  x^-_{i,n} \text{ si } n < - \langle\nu,\alpha_i\rangle  \] and the following diagrams commutate for $\epsilon, \eta$ antidominants :
        \[ \xymatrix@R+1pc@C+2pc{
    Y_{\mu + \nu} \ar[r]^{\Delta_{\mu,\nu}} \ar[d]^{\iota^{\mu + \nu}_{\epsilon,\eta}} & Y_\mu \otimes Y_\nu \ar[d]^{\iota^\mu_{\epsilon,0} \otimes \iota^\nu_{0,\eta}} \\
    Y_{\mu + \nu + \epsilon + \eta} \ar[r]^{\Delta_{\mu + \epsilon, \nu + \eta}} & Y_{\mu + \epsilon}\otimes Y_{\nu + \eta}} 
    \]
    \[ \xymatrix@R+1pc@C+2pc{
    Y_{\mu + \epsilon + \nu} \ar[r]^{\Delta_{\mu + \epsilon,\nu}} \ar[d]^{\Delta_{\mu, \epsilon + \nu}} & Y_{\mu + \epsilon} \otimes Y_\nu \ar[d]^{ \Delta_{\mu,\epsilon} \otimes \Id} \\
    Y_{\mu} \otimes Y_{\nu + \epsilon} \ar[r]^{\Id \otimes \Delta_{\epsilon, \nu}} & Y_{\mu}\otimes Y_\epsilon \otimes Y_{\nu}} 
    \]
    \end{enumerate}
\end{theoreme}

Then, we need to study the intertwining properties introduced in \cite[5.3]{Zhang}. Let us briefly recall some notions from the representation theory of shifted Yangians. 

Fix a coweight $\mu = \sum\limits_{i \in I}m_i\omega_i^\vee$  and $V$ a $Y_\mu$-module. For a $Y_\mu$-module $V$, we consider the polynomial space $V[z]$. It is a module over $Y_\mu[z]$ by scalar extension. Taking its pullback along
\[ \begin{array}{ccccc}
     \tau_z &: & Y_\mu[z] &\to & Y_\mu[z]   \\
     & & x_{i,p}^\pm & \mapsto & \sum\limits_{n\in\N}\binom{p}{n} x_{i,p-n}z^n \\  
     & & \xi_{i,p}^\pm & \mapsto & \sum\limits_{n\in\N}\binom{p}{n} \xi_{i,p-n}z^n 
\end{array}\]
gives a $Y_\mu[z]$-module, denoted $V_z$ and called deformed module. The parameter $z$ is called the spectral parameter.

Let us briefly recall the basic properties of the representation theory of shifted Yangians from \cite{Zhang}. The set $\mathcal{L}_\mu$  of \textbf{$\ell$-weight of coweight $\mu$} is the set of $I$-tuple $\textbf{P} = (P_i(z))_{i\in I} \in \C((z^{-1}))^I$ where the leading term of each $P_i(z)$ is $z^{m_i}$. One can define the associated Verma module $M(\textbf{P)}$ as the $Y_\mu$-module generated by a vector $v_0$ such that
\[ \xi_i(z) \cdot v_0 = P_i(z)v_0; \quad x_{i}^+(z) \cdot v_0 = 0 \text{ for } i \in I.\] 
It admits a unique irreducible quotient, denoted $L(\textbf{P})$ and called \textbf{highest $\ell$-weight representation}. For $i \in I$, define the \textbf{$i$-th prefundamental module} $L_i$ as the highest $\ell$-weight module of $Y_\mu$ whose $\ell$-weight is $\Psi_i := \po z^{\delta_{ij}} \pf_{j\in I}$.

Now, let us follow the constructions of \cite{Zhang}, with $\mu = \nu = 0$ and the $\ell$-weight of coweight $\zeta = \omega_i^\vee$ $\textbf{p} = \Psi_i$. For $M,N$ two root-graded $Y$module, the completed tensor product of $M$ and $N$ is defined by:
\[ M \otimes_z N := \sum\limits_{\alpha,\beta \in Q} \po \prod\limits_{\gamma \in Q_+} \po M_{(\alpha-\gamma)} \otimes N_{(\beta+\gamma)} \pf((z^{-1})) \pf.\]
It contains the $Y \otimes Y[z]$-module $M\otimes N[z]$. Note that $\Theta_i(z) \in Y\otimes_zY$.

Now, consider the algebra homomorphism
\[ \begin{array}{ccccc}
     \pi_{i,z} & : & Y_{\omega_i^\vee} & \to & \C[z]  \\
     & & x_{j}^\pm(w) & \mapsto & 0 \\
     & & \xi_j(w) &  \mapsto & (w-z)^{\delta_{ij}}
\end{array}. \]
The deformed module $L_i(z)$ is then the pullback of the regular representation $\C[z]$ along this homomorphism.

We introduce two algebra homomorphisms $F_i$ and $G_i$ from $Y_{\omega_i^\vee}$ to $Y[z]$:
\begin{gather*}
    F_i := \po \pi_{i,z} \otimes \Id \pf \Delta_{\omega_i^\vee,0}; \quad G_i := \po \Id \otimes \pi_{i,z} \pf \Delta_{0,\omega_i^\vee}
\end{gather*}
where their actions on the generators of $Y$ are given by \cite[Example 3.4]{Zhang}:
\begin{gather*} 
F_i: x_j^+(w) \mapsto (w-z)^{\delta_{ij}}x_j^+(w); \quad x_j^-(w) \mapsto x_j^-(w); \quad \xi_j(w) \mapsto (w-z)^{\delta_{ij}}\xi_j(w)\\
G_i: x^+_j(w) \mapsto x_j^+(w); \quad x_j^{-}(w) \mapsto (w-z)^{\delta_{ij}}x_j^-(w); \quad \xi_j(w) \mapsto (w-z)^{\delta_{ij}}\xi_j(w).
\end{gather*}

Finally, consider two root graded $Y$-modules $M,N$. Taking the pullback of $M\otimes N[z]$ along the homomorphism $(G_i \otimes \Id) \Delta_{\omega_i^\vee,0}$ gives a $Y_{\omega_i^\vee}$-module $\po M\otimes (L_i)_z \pf \otimes N$ and taking the pullback of $M\otimes_zN$ along the same homomorphism gives a $Y_{\omega_i^\vee}$-module that is a completion of the previous one, denoted $\overline{\po M \otimes (L_i)_z\pf \otimes N}$. Similarly, with the homomorphism $(\Id\otimes F_i)\Delta_{0,\omega_i^\vee}(x)$, one can construct the completion $ \overline{M\otimes \po (L_i)_z \otimes N \pf}$. According to \cite[theorem 5.6]{Zhang}, the action of $\Theta_i(z)$ on $M \otimes_z N$ defines a morphism of $Y_{\omega_i^\vee}$-modules between $\overline{\po M \otimes (L_i)_z\pf \otimes N} \to \overline{M\otimes \po (L_i)_z \otimes N \pf}$ where an element $ x \in Y_{\omega_i^\vee}$ acts on both modules by:
\[ \po G_i \otimes \Id  \pf \Delta_{\omega_i^\vee,0}(x) \text{ and } (\Id\otimes F_i)\Delta_{0,\omega_i^\vee}(x) \]
respectively.

We check the commutation of $\Theta_i(z)$ with the actions of the generators $x_{j,0}^+$ for $j \in I$. If $i \neq j$, then the coproducts are given by
\[ \Delta_{\omega_i^\vee,0}(x_{j,0}^+) = \Delta_{0,\omega_i^\vee}(x_{j,0}^+) = x_{j,0}^+\otimes 1 + 1 \otimes x_{j,0}^+. \]
Then, noticing that $F_i$ and $G_i$ let the generators $x^+_{j}(z)$ invariant for $j \neq i$, the commutation relation can be written
\[ \co x_{j,0}^+ \otimes 1 + 1 \otimes x_{j,0}^+, \Theta_i(z) \cf = 0.\]

We want to compute $\Delta_{\omega_i^\vee,0}(x_{i,0}^+)$ and $\Delta_{0,\omega_i^\vee}(x_{i,0}^+)$. From our definition of the coproduct, it is clear that $\Delta_{0,\omega_i^\vee}(x_{i,0}^+) = 1 \otimes x_{i,0}^+ + x_{i,0}^+ \otimes 1$. For the second coproduct, we use a standard zigzag argument. The commutation properties of the definition of the shifted coproducts give, with $\mu = \omega_i^\vee;   \nu = \eta = 0$ and $\epsilon = -\omega_i^\vee$:
\[ \po \iota_{-\omega_1^\vee,0}^{\omega_1^\vee} \otimes \iota^0_{0,0} \pf \Delta_{\omega_1^\vee,0}(x_{i,0}^+) = \Delta \po \iota_{-\omega_1^\vee,0}^{\omega_1^\vee}(x_{i,0}^+) \pf.\]
But we know $\iota_{-\omega_1^\vee,0}^{\omega_1^\vee}(x_{i,0}^+) = x_{i,1}^+$. From \cite{GNW}, we have:
\[ \Delta(x_{i,1}^+) = x_{i,1}^+ \otimes 1 + 1 \otimes x_{i,1}^+ + h_{i,0} \otimes x_{i,0}^+ - \sum\limits_{j=i+2}^n E_{j,i+1} \otimes E_{ij} + \sum\limits_{k=1}^{i-1} E_{ik}\otimes E_{k,i+1}.\]
Moreover, one can check:
\[ \po \iota_{-\omega_1^\vee,0}^{\omega_1^\vee} \otimes \iota^0_{0,0} \pf  \po x_{i,0}^+\otimes 1 \otimes x_{i,1}^+ + \xi_{i,-1} \otimes x_{i,0}^+ + \sum\limits_{j=1}^{i-1} E_{i,j}\otimes E_{j,i+1} - \sum\limits_{k=i+2}^{n+1}E_{k,i+1} \otimes E_{i,k} \pf = \Delta(x_{i,1}^+).\]
Then, by the injectivity of the shift morphisms:
\[
\Delta_{\omega_i^\vee,0}(x_{i,0}^+) = x_{i,0}^+\otimes 1 \otimes x_{i,1}^+ + \xi_{i,-1} \otimes x_{i,0}^+ + \sum\limits_{j=1}^{i-1} E_{i,j}\otimes E_{j,i+1} - \sum\limits_{k=i+2}^{n+1}E_{k,i+1} \otimes E_{i,k}.\]

So, the generator $x_{i,0}^+$ acts on both modules of \cite[Theorem 5.6]{Zhang} by
\[ x_{i,0}^+  \otimes 1 + 1\otimes x_{i,1}^+ - z \otimes x_{i,0}^+ - \sum\limits_{k=i+1}^nE_{k,i+1}\otimes E_{i,k} + \sum\limits_{j= 1}^{i} E_{i,j} \otimes E_{j,i+1}  \] and
\[ x_{i,0}^+ \otimes 1 + 1\otimes x_{i,1}^+ -  z(1 \otimes x_{i,0} ^+). \]

Since $\Theta_i(z)$ is a module homomorphism, it has to commute with these actions. Then:
\begin{align*} \Theta_i(z) \po  x_{i,0}^+  \otimes 1 + 1\otimes x_{i,1}^+ - z \otimes x_{i,0}^+ - \sum\limits_{k=i+1}^nE_{k,i+1}\otimes E_{i,k} + \sum\limits_{j= 1}^{i} E_{i,j} \otimes E_{j,i+1} \pf \\= \po x_{i,0}^+ \otimes 1 + 1\otimes x_{i,1}^+ -  z (1 \otimes x_{i,0} ^+) \pf \Theta_i(z)
\end{align*}
which can be written as the second formula in proposition \ref{intertwiningyangian}.

Finally, to conclude with the uniqueness of the solution of this system with $\Theta_{i,0} = 1$, we use an argument similar to \cite[section 4.2]{GWT} and \cite[theorem 5.7, step 2]{Zhang}. Note that by projecting the previous equation on $Y_{\alpha_i -\beta}\otimes Y_{\beta}$ for $\beta \in Q_+$, one obtains for $i\neq j$:
\[ \co x_{j,0}^+ \otimes 1 + 1 \otimes x_{j,0}^+, \Theta_{i,\beta}(z) \cf  = 0 \]
and, on the other hand:
\begin{align*} \co x_{i,0}^+ \otimes 1, \Theta_{i,\beta}(z) \cf = &- \co 1\otimes x_{i,1}^+ -z(1\otimes x_{i,0}^+), \Theta_{i,\beta-\alpha_i}(z)\cf + \sum\limits_{j=1}^{i-1} \Theta_{i,\beta - \gamma_j - \alpha_i}(z) \po E_{i,j}\otimes E_{j,i+1}\pf \\ &- \sum\limits_{k = i+2}^{n+1} \Theta_{i,\beta - \eta_k - \alpha_i } (z)\po E_{k,i+1} \otimes E_{i,k} \pf + \Theta_{i,\beta- \alpha_i} (z) \po h_{i,0} \otimes x_{i,0}^+ \pf
\end{align*}
where $\gamma_j = \alpha_j + \dots + \alpha_{i-1}$ and $\eta_k = \alpha_{i+1} + \dots + \alpha_k$. Define the height of a weight $\beta$ by $h(\beta) = \langle \beta,\sum\limits_{i\in I}\omega_i^\vee \rangle \in \N$. Then, by induction on $h(\beta)$, $\Theta_i(z)$ is uniquely determined by the fact that $\Theta_{i,0}(z) = 1\otimes 1$.

\bibliography{biblio.bib}
\bibliographystyle{plain}
\end{document}